\documentclass[12pt]{amsart}

\usepackage{graphicx}
\usepackage{color}
\usepackage{hyperref}
\usepackage{cite}
\usepackage[margin=1in]{geometry}
\usepackage{enumerate}
\usepackage{tikz}
\usetikzlibrary{arrows,automata,positioning}

\newtheorem{theorem}{Theorem}
\numberwithin{theorem}{section}
\newtheorem{lemma}[theorem]{Lemma}
\newtheorem{proposition}[theorem]{Proposition}
\newtheorem{corollary}[theorem]{Corollary}

\theoremstyle{definition}
\newtheorem{definition}[theorem]{Definition}

\newtheorem{problem}[theorem]{Problem}
\theoremstyle{remark}

\numberwithin{equation}{section}

\DeclareMathOperator{\sgn}{sgn}
\DeclareMathOperator{\degree}{deg}
\DeclareMathOperator{\dist}{dist}

\newcommand{\basilgroup}{B}
\newcommand{\tlambda}{\tilde{\lambda}}

\begin{document}
	
\title{Spectral properties of graphs associated to the Basilica group}

\author[Brzoska, George, Jarvis, Rogers, Teplyaev]{Antoni Brzoska, Courtney George, Samantha Jarvis, Luke~G. Rogers, Alexander Teplyaev}

\address{Department of Mathematics,
	University of Connecticut,
	Storrs, CT 06269-1009 USA}          
\thanks{Research supported in part by NSF grants 1659643  and 1613025.}                          
%\curraddr{...}                                          
\email{\ 
	\newline
antoni.brzoska@uconn.edu
\newline
courtney.george@ucr.edu
%courtney.george@uky.edu
\newline
sjarvis@gradcenter.cuny.edu
\newline
luke.rogers@uconn.edu
\newline
alexander.teplyaev@uconn.edu
	}          
\urladdr{https://math.uconn.edu/person/luke-rogers/	
\newline
https://math.uconn.edu/person/alexander-teplyaev/	
	}
\begin{abstract}
We provide the foundation of the spectral analysis of the Laplacian on the orbital Schreier graphs of the Basilica group, the iterated monodromy group   of the quadratic polynomial $z^2-1$. This group is an  important example in the class of self-similar amenable but not elementary amenable finite automata groups studied by Grigorchuk, \.Zuk, \v Suni\'c, Bartholdi, Vir\'ag, Nekrashevych, Kaimanovich, Nagnibeda et al. We prove that the spectrum of the Laplacian has infinitely many gaps and that the support of the KNS Spectral Measure is a Cantor set. Moreover, on a generic blowup, the spectrum coincides with this Cantor set, and is pure point with localized eigenfunctions and eigenvalues located at the endpoints of the gaps. 
\end{abstract} 
\date{\today}
\maketitle

\section{Introduction}

The Basilica group is a well studied example of a self-similar automata group. 
It has interesting algebraic properties, 
for which we refer to the work of  Grigorchuk and \.Zuk, who introduced the group  in~\cite{MR1902367} and studied some of its spectral properties in~\cite{MR1929716}, and of Bartholdi and Vir\'ag~\cite{MR2176547}, who proved that it is amenable but not sub-exponentially amenable.  However the spectral properties of the Basilica group do not seem to be fully accessible by using the techniques introduced in the foundational papers~\cite{BG2000a,BG2000b}.  By the work of Nekrashevych~\cite{Nekbook} the Basilica group is an iterated monodromy group and has as its limit set the Basilica fractal, which is the Julia set of $z^2-1$.  The resistance form and Laplacian on this fractal were introduced and studied in~\cite{RogTep}, where it was proved that the spectral dimension $d_s$ of the Basilica fractal is equal to~$\frac43$. 
In this paper we combine an array of tools from various areas of mathematics to study the spectrum of the orbital Schreier graphs of the Basilica group. Our work is strongly motivated by recent results of Grigorchuk, Lenz,  and Nagnibeda, see \cite[and references therein]{GLN1,GLN2}. Our results are closely related to the new substantial work 
\cite{dang2020self} by Dang,  Grigorchuk,   and Lyubich. In particular, our Corollary~\ref{cor-key} should be compared to 
\cite[Remark 1.3]{dang2020self} and, we hope, will provide a foundation for further study related to the recent preprints \cite{grigorchuk2020integrable,grigorchuk2020spectra}. 

As for self-similar groups in general, a great deal of the analysis of the Basilica group rests on understanding the structure of its Schreier graphs and their limits.     
Many  properties of such graphs were obtained by D'Angeli, Donno, Matter and Nagnibeda~\cite{nagnibeda}, 
including a classification of the orbital Schreier graphs, which are limits of finite Schreier graphs in the pointed Gromov-Hausdorff sense.
In the present work we consider spectral properties of some graphs obtained by a simple decomposition of the Schreier graphs.  
These graphs may still be used to analyze most orbital Schreier graphs.  

Our main results include construction of a dynamical system for the spectrum of the Laplacian on Schreier graphs that gives an explicit formula for the multiplicity of eigenvalues and a geometric description of the supports of the corresponding eigenfunctions, associated formulas for the proportion of the KNS spectral measure on orbital Schreier graphs that is associated to eigenvalues for each of the finite approximation Schreier graphs, and a proof that the spectra of orbital Schreier graphs contain infinitely many gaps and no intervals.  We also show that the Laplacian spectrum for a large class of orbital Schreier graphs is pure point. 

The paper is arranged as follows:
\begin{itemize}
	\item \mbox{In Section~\ref{sec-Gamma-Gn}} we introduce the Basilica group, its Schreier graphs $\Gamma_n$ and their Laplacians. We then make a simple decomposition of $\Gamma_n$ to introduce graphs  $G_n$ which will be more tractable in our later analysis. The main result of Section~\ref{sec-Gamma-Gn}, Theorem~\ref{thm:orbitalSareblowups}, is that moving from $\Gamma_n$ to $G_n$ is of little significance for the limiting structures.  Specifically we show that, with one exception, all isomorphism classes of orbital Schreier graphs of the Basilica group are also realized as infinite blowups of the graphs $G_n$.  Conversely, all blowups of $G_n$, except those with boundary points, are orbital Schreier graphs of the Basilica group.

\item
\mbox{In Section~\ref{section:dynamics}} we give a dynamical description of the spectrum of $G_n$ which reflects the self-similarity in its construction. It should be noted that a different dynamical system for the spectrum of the Basilica group was obtained some time ago in~\cite{MR1929716} by another method, but we do not know whether it is possible to do our subsequent analysis for that system. Subsection~\ref{subseccharpoly}  introduces our first recursion for characteristic polynomials of the Laplacian. 
Subsection~\ref{subsec:locefns} describes localized eigenfunctions and Theorem~\ref{thm:factorizationofcn} provides a factorization of the characteristic polynomial for the $G_n$ Laplacian. 
In particular, Theorem~\ref{thm:factorizationofcn} counts eigenvalues that are introduced in earlier levels of the construction of the structure, and describe  their multiplicities by using geometric features of the graphs. These geometric features represent local symmetries and correspond to the number of ``copies'' of localized and non-localized eigenfunctions. 
The recursive dynamics of these factors is considered in more detail in Subsection~\ref{subsec-dyn-gamma-n}, where we find in Corollary~\ref{cor:rootdynamics} that a vastly simpler dynamics is valid for a rational function having roots at the eigenvalues for $G_n$ that are not eigenvalues of any earlier $G_k$, $k<n$, and poles at the latter values with specified multiplicities.  This simpler dynamics is crucial in our later work because it is susceptible to a fairly elementary and direct analysis.
  
\item In  Section~\ref{sec-KMS}, Theorem~\ref{thm:howtogetmostofspect}, we prove an approximation result for the Kesten--von-Neuman--Serre (KNS) spectral measure of a blowup $G_\infty$ of the graphs $G_n$, which is a version of the integrated density of states. For details of this measure we refer to~\cite{GZ04}. 

\item
\mbox{In Section~\ref{section:gaps}} we prove the existence of {\em gaps}, which are intervals that do not intersect the spectum of the Laplacian for any of the graphs $G_n$, and show that for each $\lambda$ in the spectrum of the Laplacian for some $G_n$ there are a sequence $k_j$ and spectral values for the Laplacian on $G_{n+2k_j}$ that accumulate at $\lambda$,  see Theorem~\ref{thm:valuesandgaps}.  It follows readily that the support of the KNS spectral measure is a Cantor set.

\item
\mbox{In Section~\ref{sec-pp}} we use the approach developed in \cite{Teplyaevthesis,MT} to show that a generic set of  blowups of the graphs $G_n$, or equivalently a generic set of orbital Schreier graphs, have pure point spectrum, see Theorem~\ref{thm:pureptspect}. It follows that the spectrum of the natural Markov operator on the blowup, which is sometimes called the Kesten spectrum, coincides with the Cantor set that forms the support of the KNS spectral measure.
\end{itemize}

The motivation for our work comes from three sources. 
First, we are interested to develop methods that provide more information about certain self-similar groups, see the references given above and \cite{Kaimanovich09,Kaimanovich05,Kaimanovich03,NT,BKM}. 
Second, we are interested to develop new methods in spectral analysis on fractals. 
Our work gives one of the first results available in the literature 
that gives precise information about the spectrum of a graph-directed self-similar structure, 
making more precise the asymptotic analysis in \cite{HN}. 
For related results in self-similar setting, see  \cite{Teplyaevthesis,vib,HSTZ,MR3725301,MR2746525,MR2608413,MR2727362,MR2448573,MR2150975,MR1990573,MR1284792,MR1245223,Shima,j3,j3,j1,Sabot1,Sabot2,Sabot3,Sabot4}.
One can hope that spectral analysis of the Laplacian on  Schreier graphs in some sense can provide a basis for harmonic analysis on self-similar groups, following ideas of \cite{Sharmonic,W}. 
Third, our motivation comes from the  works in physics and probability dealing with various spectral oscillatory phenomena  
\cite[and references therein]{A,D,G,k1,k2,k3,ADT}.  
In general terms, our results are a part of the study of the systems with aperiodic order, 
see \cite[and references therein]{d3,d2,d1,AJ}.

\subsection*{Acknowledgments}
The last two authors thank 
Nguyen-Bac Dang, Rostislav Grigorchuk,  Mikhail Lyubich,
Volodymyr Nekrashevych, 
Tatiana Smirnova-Nagnibeda, and 
Zoran \v Suni\'c  
for helpful and interesting discussions. 

\section{The graphs $\Gamma_n$ and $G_n$ and their Laplacians}\label{sec-Gamma-Gn}

\subsection{The Basilica group and its Schreier graphs}

Let $T$ be the binary rooted tree.  We  write its vertices as finite words $v\in\{0,1\}^*:=\cup_{n=0}^\infty \{0,1\}^n$;  a vertex $v=v_1\dotsm v_n$ is said to be of level $n$, and by convention $\{0,1\}^0=\{\emptyset\}$ is the null word. The edges containing the vertex $v=v_1\dotsm v_n$ go to the children $v0$, $v1$ and the parent $v_1\dotsm v_{n-1}$.  Evidently a tree automorphism of $T$ preserves the levels of vertices.  The set of right-infinite words, which may be considered to be the boundary of $T$, is written as $\{0,1\}^\omega=\partial T$.

The Basilica group is generated by an automaton.  There is a rich theory of automata and automatic groups, for which we refer to the expositions in~\cite{Nekbook,BartGrigNek}.  For the Basilica the automaton is a quadruple consisting of a set of states $\mathcal{S}=\{e,a,b\}$ (where $e$ means identity), the alphabet $\{0,1\}$, a transition map $\tau:\mathcal{S}\times\{0,1\}\to\mathcal{S}$ and an output map $\rho:\mathcal{S}\times\{0,1\}\to\{0,1\}$.  It is standard to present the automaton by using a Moore diagram, given in Figure~\ref{Moorediag}, which is a directed graph with vertex set $\mathcal{S}$ and arrows for each $(s,j)$, $j\in\{0,1\}$ that point from $s$ to $\tau(s,j)$ and are labelled with $j|\rho(s,j)$.

\begin{figure}
\centering
\begin{tikzpicture}[->, >=stealth', shorten >=1pt,
auto,node distance=1cm and 5cm,on grid,semithick,
inner sep=2pt,bend angle=45]
\node[state] (id) {$id$};
\node[state] (a) [above left=of id] {$a$};
\node[state] (b) [below left = of id] {$b$};
\path[every node/.style={font=\footnotesize}]
	(a) 	edge[bend right] node[left] {$0|0$} (b)
		edge[bend left] node {$1|1$} (id)
	(b)	edge[bend right] node[right] {$0|1$}  (a)
		edge[bend right] node[below] {$1|0$} (id)
	(id)   edge[loop right] node {$0|0$ and $1|1$} ();
\end{tikzpicture}
\caption{The Moore diagram for the Basilica group automaton.}
\label{Moorediag}
\end{figure}
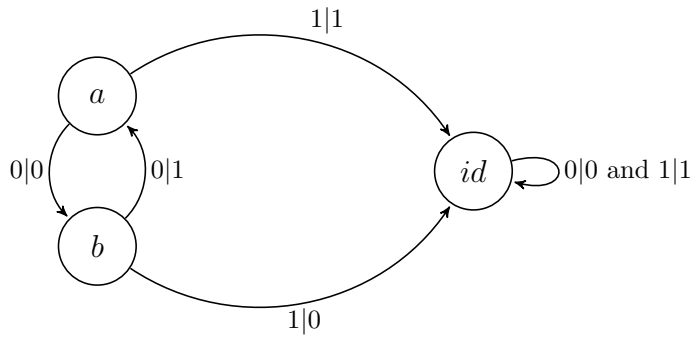

The automaton defines, for each $s\in\mathcal{S}$, self maps $\mathcal{A}_s$ of $\{0,1\}^*$ and $\{0,1\}^\omega$ (i.e. $T$ and $\partial T$) by reading along the word from the left and altering one letter at a time. Specifically, given a state $s$ and a word $v=v_1v_2v_3\dotsm$ (which may be finite or infinite), the automaton ``reads'' the letter $v_1$, writes $\rho(s,v_1)$, moves one position to the right and ``transitions'' to state $\tau(s,v_1)$, which then reads $v_2$, and so forth. Observe that these $\mathcal{A}_s$ are tree automorphisms of $T$.  The Basilica group is the group of automorphisms of $T$ generated by the $\mathcal{A}_s$ with $s\in\mathcal{S}$.

Classically, a Schreier graph of a group $\basilgroup$  is defined from a generating set $S$ and a subgroup $H$ by taking the vertices to be the left cosets $\{gH:g\in\basilgroup\}$ and the edges to be of the form $(gH,sgH)$ for $s\in S$.  In the case that $\basilgroup$ acts transitively on a set $\tilde{T}$ one takes $H$ to be the stabilizer subgroup of an element; this subgroup depends on the element, but the Schreier graphs are isomorphic.  Moreover, one may then identify cosets  of $H$ with elements of $\tilde{T}$, at which point the Schreier graph can be thought to have vertex set $\tilde{T}$ and edges $\bigl\{\{v,sv\}:v\in\tilde{T},s\in S\setminus\{e\}\bigr\}$.  Note that we remove the identity from $S$ to avoid unnecessary loops, and that the Schreier graphs considered in this paper have undirected edges.

The Basilica group is transitive on levels of the binary tree $T$, so we may define a Schreier graph for each level by the above construction.  Removing the identity from $\mathcal{S}$ we take the generating set to be $S=\{A_a,A_b\}$. More precisely, the $n^\text{th}$~Schreier graph $\Gamma_n$ of the Basilica group has vertices the words $\{0,1\}^n$ and (undirected) edges between pairs of words $w$, $w'$ for which $\mathcal{A}_a(w)=w'$ or $\mathcal{A}_b(w)=w'$; it is often useful to label the edge with $a$ or $b$ to indicate the associated generator.

The action of $\basilgroup$ on the boundary $\partial T$ is not transitive, but for each $v\in\partial T$ we may take the Schreier graph defined on the orbit of $v$, which is just that of the stabilizer subgroup of $\basilgroup$ at $v$. This is called the orbital Schreier graph $\Gamma_v$.  If the length $n$ truncation of $v$ is denoted  $[v]_n$  then the sequence of pointed finite Schreier graphs $(\Gamma_n,[v]_n)$ converges in the pointed Gromov-Hausdorff topology to $(\Gamma_v,v)$.  One description of this convergence is to define the distance between pointed graphs $(\Gamma',x'),(\Gamma'',x'')$ as follows:
\begin{equation}\label{eq:pGHdist}
\dist_{pGH}\bigl((\Gamma',x'),\Gamma'',x'')\bigr)= \inf \Bigl\{ \frac1{r+1}:  B_{\Gamma'}(x',r) \text{ is graph isomorphic to } B_{\Gamma''}(x'',r)\Bigr\}.
\end{equation}
A classification of the orbital Schreier graphs of the Basilica group is one main result of~\cite{nagnibeda}.

It is helpful to understand the relationship between the Schreier graphs for different levels.  To see it, we compute for a finite word $w$ that $a(1w)=1e(w)=1w$ and $a(0w)=0b(w)$, while  $b(1w)=0e(w)=0w$  and $b(0w)=1a(w)$.  This says that at any word beginning in $1$ there is an $a$-self-loop and every pair $\{1w,0w\}$ is joined by a $b$-edge.  It also says that if there is a $b$-edge $\{w,b(w)\}$ at scale $n$ then there is an $a$-edge $\{0w,0b(w)\}$ at scale $(n+1)$, if there is an $a$-edge $\{w,a(w)\}$ at scale $n$ there is a $b$-edge $\{0w,1a(w)\}$ at scale $n+1$, and if there is an $a$-loop at~$w$ there are two $b$-edges between~$0w$ and~$1w$.  With a little thought one sees that these may be distilled into a set of replacement rules for obtaining $\Gamma_{n+1}$ from $\Gamma_n$.  Each $b$-edge in $\Gamma_n$ becomes an $a$-edge in $\Gamma_{n+1}$, an $a$-loop at~$1w$ becomes two $b$-edges between~$01w$ and~$11w$, and an $a$-edge, which can only be between words $0w,0b(w)$, becomes $b$-edges from $10b(w)$ to both $00w$ and $00b(w)$; $a$-loops are also appended at words beginning in~$1$.  These replacement rules are summarized in Figure~\ref{replacementwholeSchreier} and may be used to construct any $\Gamma_n$ iteratively, beginning with $\Gamma_1$, which is shown along with $\Gamma_2$ and $\Gamma_3$ in Figure~\ref{fig:gammagraphs}.  For a more detailed discussion of these rules see  Proposition~3.1 in \cite{nagnibeda}.

\begin{figure}
\centering
\begin{tikzpicture}
\draw   (0,2) node[circle,fill,inner sep=2pt,label=below:$w$](tlb){} -- node[above]{$b$}  (3,2) node[circle,fill,inner sep=2pt,label=below:$z$](trb){};
\draw  (5,2) node[circle,fill,inner sep=2pt,label=below:$0w$](tla){} -- node[above]{$a$}  (8,2) node[circle,fill,inner sep=2pt,label=below:$0z$](tra){};
\draw   (0,0) node[circle,fill,inner sep=2pt,label=below:$0w$](bla){} -- node[above]{$a$}  (3,0) node[circle,fill,inner sep=2pt,label=below:$0z$](bra){};
\path  (5,0) node[circle,fill,inner sep=2pt,label=below:$00w$](blbb){}  (6.5,0) node[circle,fill,inner sep=2pt,label=above:$10z$](bmbb){}  (8,0) node[circle,fill,inner sep=2pt,label=below:$00z$](brbb){};
\draw (blbb)--node[above]{$b$} (bmbb);
\draw (bmbb)--node[above]{$b$} (brbb);
\path (1.5,-2) node[circle,fill,inner sep=2pt,label=below:$1w$](a){} (6.5,-2) node[circle,fill,inner sep=2pt,label=below:$11w$](abb){} (8,-2) node[circle,fill,inner sep=2pt,label=below:$01w$](bb){} ;
\draw (abb) to [bend right] node[below]{$b$} (bb);
\draw (abb) to [bend left] node[above]{$b$} (bb);
\path[-, every loop/.style= {looseness=10, distance=20, in=300,out=240}]   (bmbb) edge [loop below]  node{$a$} ();
\path[-, every loop/.style= {looseness=10, distance=20, in=150,out=210}]   (a) edge [loop left]  node{$a$} ();
\path[-, every loop/.style= {looseness=10, distance=20, in=150,out=210}]   (abb) edge [loop left]  node{$a$} ();

\path[->,thick] (3.5,2) node{} edge [bend left] (4.5,2) node{};
\path[->,thick] (3.5,0) node{} edge [bend left] (4.5,0) node{};
\path[->,thick] (3.5,-2) node{} edge [bend left] (4.5,-2) node{};
\end{tikzpicture}
\caption{Replacement Rules for $\Gamma_n$}
\label{replacementwholeSchreier}
\end{figure}
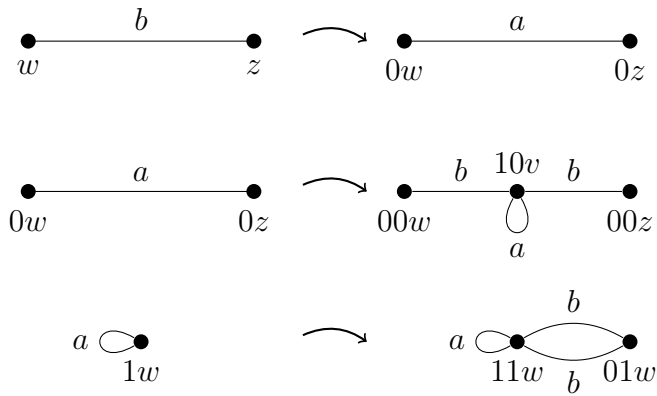

\begin{figure}
\centering
\begin{tikzpicture}
\path  (5,4) node[circle,fill,inner sep=2pt,label=below:$0$](0){} edge [loop left,every loop/.style= {looseness=10, distance=20,in=150,out=210}] node{$a$} () --  (6.5,4) node[circle,fill,inner sep=2pt,label=below:$1$](1){} edge [loop right,every loop/.style= {looseness=10, distance=20,in=30,out=330}] node{$a$}();
\draw (0)  to [bend left] node[above] {$b$} (1);
\draw (0)  to [bend right] node[below] {$b$} (1);
\path  (3.5,2) node[circle,fill,inner sep=2pt,label=below:$10$](10){} edge [loop left,every loop/.style= {looseness=10, distance=20,in=150,out=210}] node{$a$} () --  (5,2) node[circle,fill,inner sep=2pt,label=below:$00$](00){}  --  (6.5,2) node[circle,fill,inner sep=2pt,label=below:$01$](01){}  (8,2) node[circle,fill,inner sep=2pt,label=below:$11$](11){} edge [loop right,every loop/.style= {looseness=10, distance=20,in=30,out=330}] node{$a$}();
\draw (10)  to [bend left] node[above] {$b$} (00);
\draw (10)  to [bend right] node[below] {$b$} (00);
\draw (00)  to [bend left] node[above] {$a$} (01);
\draw (00)  to [bend right] node[below] {$a$} (01);
\draw (01)  to [bend left] node[above] {$b$} (11);
\draw (01)  to [bend right] node[below] {$b$} (11);
\path[yshift=-25]  (2,0) node[circle,fill,inner sep=2pt,label=below:$110$](110){} edge [loop left,every loop/.style= {looseness=10, distance=20,in=150,out=210}] node{$a$} () --  (3.5,0) node[circle,fill,inner sep=2pt,label=below:$010$](010){} --  (5,0) node[circle,fill,inner sep=2pt,label=below:$000$](000){}  -- (5.75,1)node[circle,fill,inner sep=2pt,label=left:$100$](100){} edge [loop above,every loop/.style= {looseness=10, distance=20,in=60,out=120}] node{$a$} ()  -- (5.75,-1) node[circle,fill,inner sep=2pt,label=left:$101$](101){} edge [loop below,every loop/.style= {looseness=10, distance=20,in=240,out=300}] node{$a$} ()  --  (6.5,0) node[circle,fill,inner sep=2pt,label=below:$001$](001){}  (8,0) node[circle,fill,inner sep=2pt,label=below:$011$](011){} -- (9.5,0) node[circle,fill,inner sep=2pt,label=below:$111$](111){} edge [loop right,every loop/.style= {looseness=10, distance=20,in=30,out=330}] node{$a$}();
\draw (110)  to [bend left] node[above] {$b$} (010);
\draw (110)  to [bend right] node[below] {$b$} (010);
\draw (010)  to [bend left] node[above] {$a$} (000);
\draw (010)  to [bend right] node[below] {$a$} (000);
\draw (001)  to [bend left] node[above] {$a$} (011);
\draw (001)  to [bend right] node[below] {$a$} (011);
\draw (011)  to [bend left] node[above] {$b$} (111);
\draw (011)  to [bend right] node[below] {$b$} (111);
\draw (000)  to [bend left] node[above] {$b$} (101);
\draw (101)  to [bend left] node[above] {$b$} (001);
\draw (001)  to [bend left] node[right] {$b$} (100);
\draw (100)  to [bend left] node[left] {$b$} (000);
\end{tikzpicture}
\caption{The graphs $\Gamma_1,\Gamma_2$ and $\Gamma_3$.}
\label{fig:gammagraphs}
\end{figure}
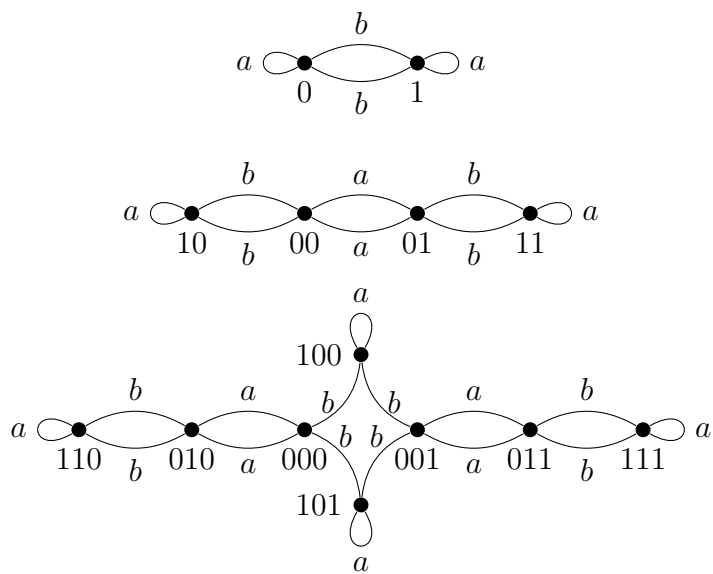

\subsection{The graphs $G_n$}

In order to simplify some technicalities in the paper we do not work directly with the graphs $\Gamma_n$ but instead treat  graphs $G_n$ defined as follows.   For $n\geq 2$, replace the degree four vertex $0^n$ in $\Gamma_n$ with four vertices, one for each edge incident upon $0^n$, and  call these boundary vertices.  Observe that this produces  two new graphs, each with two boundary vertices.   Denote the larger subgraph by $G_n$ and observe that the self-similarity of $\Gamma_n$ implies the smaller subgraph is isomorphic to $G_{n-1}$ if $n\geq 3$.   By using the addressing scheme for the finite Schreier graphs, the subgraph $G_n$ consists of those vertices in $\Gamma_n$ with addresses not ending in~$10$, plus the boundary vertices.  Evidently one can recover the graph $\Gamma_n$ by identifying the boundaries of $G_n$ and $G_{n-1}$ as a single point; we return to this idea later and illustrate it for $n=3$ in Figure~\ref{add}????.  To ensure this is true for all $n\geq1$ we define $G_0$ and $G_1$ as in Figure~\ref{approx}, which also shows $G_2$ and $G_3$.  Then it is apparent we may generate the graphs $G_n$ from $G_0$  using the same replacement rules for $\Gamma_n$ that are depicted in Figure~\ref{replacementwholeSchreier}.  We denote the set of  boundary points of $G_n$ by $\partial G_n$.

\begin{figure}
\centering
\begin{tikzpicture}
\path  (2,6)node{$G_0$} (3,6) node[circle,fill,inner sep=2pt](0){} --node[above]{$a$}   (6,6) node[circle,fill,inner sep=2pt](1){};
\draw (0) to (1);
\path  (8,6) node{$G_2$} (9,6) node[circle,fill,inner sep=2pt](200a){} --node[above]{$a$}  (10.5,6) node[circle,fill,inner sep=2pt](2u){} --node[above]{$a$}  (12,6) node[circle,fill,inner sep=2pt](200b){};
\draw (200a) to (2u) to (200b);
\path[-, every loop/.style= {looseness=10, distance=20, in=300,out=240}]   (2u)  --   (10.5,5) node[circle,fill,inner sep=2pt](211){} edge [loop below] node {$a$} (10.5,5) ;
\draw (2u) edge[bend right] node[left]{$b$} (211) edge[ bend left] node[right]{$b$} (211);
\path  (2,4.5) node{$G_1$} (3,4.5) node[circle,fill,inner sep=2pt](00a){} --node[above]{$b$}  (4.5,4.5) node[circle,fill,inner sep=2pt](u){} --node[above]{$b$}  (6,4.5) node[circle,fill,inner sep=2pt](00b){};
\draw (00a) to (u) to (00b);
\path[-, every loop/.style= {looseness=10, distance=20, in=300,out=240}]   (u) edge [loop below] node{$a$}();
\draw  (2,3)node{$G_3$} (4.5,3) node[circle,fill,inner sep=2pt]{} --node[above]{$b$} (6,3) node[circle,fill,inner sep=2pt](300a){} --node[above]{$b$}  (7.5,3) node[circle,fill,inner sep=2pt](3u){} --node[above]{$b$}  (9,3) node[circle,fill,inner sep=2pt](300b){} --node[above]{$b$} (10.5,3) node[circle,fill,inner sep=2pt]{};
\path[-, every loop/.style= {looseness=10, distance=20, in=300,out=240}]   (3u)  --   (7.5,2) node[circle,fill,inner sep=2pt](311){} -- (7.5,1) node[circle,fill,inner sep=2pt](3101){} edge [loop below] node {$a$}(7.5,0)  ;
\path[every loop/.style= {looseness=10, distance=20, in=300,out=240}]   (300a) edge [loop below] node {$a$} ()  (300b) edge [loop below] node {$a$} ();
\draw (3u) edge[bend right] node[left]{$a$} (311) edge[ bend left] node[right]{$a$} (311);
\draw (311) edge[bend right] node[left]{$b$} (3101) edge[ bend left] node[right]{$b$} (3101);
\end{tikzpicture}
\caption{Graphs $G_n$, $n=0,1,2,3$}
\label{approx}
\end{figure}
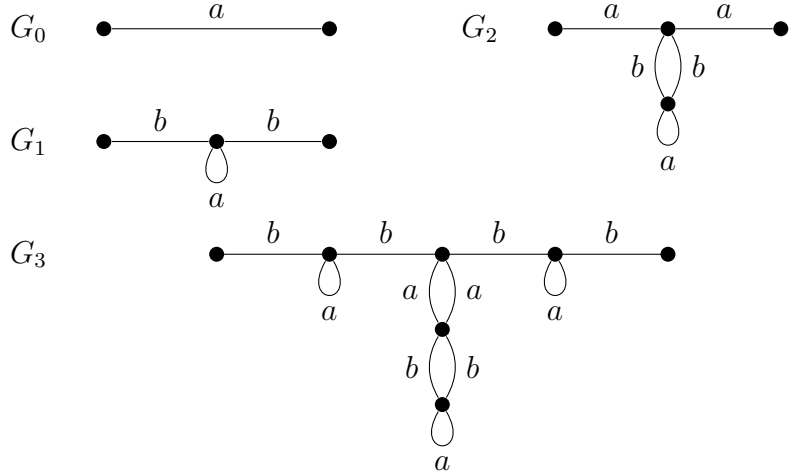

We define a Laplacian $L_n$ on $G_n$ in the usual manner.  Let $\ell^2_n$ denote the functions $\mathbb{R}^{G_n}$ with $L^2$ norm with respect to the counting measure on the vertex set.   For vertices $x,y$ of $G_n$ let $c_{xy}$ be the number of edges joining $x$ and $y$ and note that $c_{xy}\in\{0,1,2\}$.
\begin{definition}\label{def:lap} The Laplacian on $\ell^2_n$ is
\begin{equation}\label{eqn:defnofL_n}
	L_n  f(x) =  \sum_y c_{xy}(f(x)- f(y)).
	\end{equation}
\end{definition}
$L_n$ is self-adjoint,  irreducible because $G_n$ is connected, and non-negative definite because  $\sum_x f(x)L_nf(x)=\frac12\sum_{x,y} c_{xy} (f(x)-f(y))^2$.

We will also make substantial use of the Dirichlet Laplacian, which is given by~\eqref{eqn:defnofL_n} but with domain the functions $\{f\in \mathbb{R}^{G_n}:f|_{\partial G_n}=0\}$.

\subsection{Blowups of $G_n$ and their relation to Schreier graphs}

Since our graphs $G_n$ are not Schreier graphs we cannot take orbital graphs as was done in the Schreier case. A convenient alternative is a variant of the notion of fractal blowup due to Strichartz~\cite{StrichartzFractalsinthelarge}, in which a blowup of a fractal defined by a contractive iterated function system is defined as the union of images under branches of the inverses of the i.f.s.\ maps. The corresponding idea in our setting is to use branches of the inverses of the graph coverings  corresponding to truncation of words; these inverses are naturally represented by appending letters. The fact that we restrict to $G_n$ means words with certain endings are omitted.

Recall that in the usual notation for finite Schreier graphs, $G_n$, $n\geq2$, is the subset of $\Gamma_n\setminus\{0^n\}$ consisting of words that do not end with~$10$, except that the vertex~$0^n$ is replaced with two distinct boundary vertices which we will write $0^{n-1}x$ and $0^{n-1}y$; if $n\geq3$ the former is connected to a vertex ending in~$0$ and the latter to one ending in~$1$.  One definition of an infinite blowup is as follows.  
\begin{definition}\label{def:Gnblowups}
An infinite blowup of the graphs $G_n$ consists of a sequence $\{k_n\}_{n \in \mathbb{N}} \subset \mathbb{N}$ with $k_1=2$ and $k_{n+1}-k_n\in\{1,2\}$ for each $n$, and corresponding graph morphisms $\iota_{k_n}:G_{k_n}\to G_{k_{n+1}}$ of the following specific type. If $k_{n+1}-k_n=1$ then $\iota_{k_n}$ is the map that appends~$1$ to each non-boundary address and replaces both $x$ and $y$ by $01$.  If $k_{n+1}-k_n=2$ then $\iota_{k_n}$ is one of two maps: either the one that appends $00$ to non-boundary addresses and makes the substitutions $x\mapsto00x$, $y\mapsto 001$, or the one that appends $01$ to non-boundary addresses and makes the substitutions $x\mapsto001$ and $y\mapsto00y$. Now let $G_{\infty}$ be the direct limit (in the category of sets) of the system $(G_{k_n},\iota_{k_n})$.  We write $\tilde{\iota}_{k_n}:G_{k_n}\to G_\infty$ for the corresponding canonical graph morphisms.
\end{definition}
Note that the choice $k_1=2$ was made only to ensure validity of the notation for $G_n$ when definining $\iota_{k_n}$; with somewhat more notational work we could begin with $k_1=0$.

The following theorem is essentially known, see \cite{EndsofSchreiergraphs,Bondarenko,nagnibeda}. We provide a concise proof for the sake of completeness and convenience of the reader. 
\begin{theorem}\label{thm:orbitalSareblowups}
With one exception, all isomorphism classes of orbital Schreier graphs of the Basilica group are also realized as infinite blowups of the graphs $G_n$.  Conversely, all blowups of $G_n$ except those with boundary points are orbital Schreier graphs.
\end{theorem}

\begin{proof}
The orbital Schreier graph $\Gamma_v$ associated to the point $v\in\partial T$ is the pointed Gromov-Hausdorff limit of the sequence $(\Gamma_k,[v]_k)$ with the distance in~\eqref{eq:pGHdist}.  
Now set $k_1=2$ and define $k_{n+1}$ inductively by $k_{n+1}=k_n+1$ if $v_{k_n+1}=1$ and $k_{n+1}=k_n+ 2$ if $v_{k_n+1}=0$. It follows that $[v]_{k_{n+1}}$ is obtained from $[v]_{k_n}$ by appending one of $00$, $01$, or~$1$, and we can choose $\iota_{k_{n+1}}$ so $[v]_{k_{n+1}}=\iota_{k_{n+1}}([v]_n)$.   The maps $\iota_{k_n}:G_{k_n}\to G_{k_{n+1}}$ define a fractal blowup associated to the boundary point $v$ and we immediately observe that if the distance between $[v]_{k_n}$ and $0^{k_n}$ diverges as $n\to\infty$ then the sequence $(G_{k_n},[v]_{k_n})$ converges in the pointed Gromov-Hausdorff sense~\eqref{eq:pGHdist} to the limit of $(\Gamma_{k_n},[v]_{k_n})$, which is precisely the orbital Schreier graph $(\Gamma_v,v)$

In the alternative circumstance that  the distance between $[v]_{k_n}$ and $0^{k_n}$ remains bounded we determine from Proposition~2.4 of~\cite{nagnibeda} that $v$ is of the form $w\bar{0}$ or $w\overline{01}$, where $w$ is a finite word.  Moreover, in this circumstance Theorem~4.1 of~\cite{nagnibeda} establishes that $\Gamma_v$ is the unique (up to isomorphism) orbital Schreier graph with 4 ends.  Accordingly, our infinite blowups capture all orbital Schreier graphs except the one with 4 ends.

The converse is almost trivial: the definition of an infinite blowup gives a sequence $k_n$ and corresponding elements of $\{1,00,01\}$.  Appending these inductively defines an infinite word $v$ and thus an orbital Schreier graph. If $v$ is not of the form $w\bar{0}$ or $w\bar{01}$ then the orbital Schreier graph is simply $G_\infty$ with distinguished point $v$.  Otherwise the blowup is not the same as the orbital Schreier graph for the unsurprising reason that the blowup contains~$\bar{0}$ as a boundary point.
\end{proof}

\subsection{The Laplacian on a blowup}
Fix a blowup $G_\infty$ given by sequences $k_n$ and $\iota_{k_n}$ as in Definition~\ref{def:Gnblowups} and let $l^2$ denote the space of functions on the vertices of $G_\infty$ with counting measure and $L^2$ norm.
\begin{definition}\label{defn:Lapinfty}
The Laplacian $L_\infty$ on $l^2$ is defined as in~\eqref{eqn:defnofL_n} where $c_{xy}$ is the number of edges joining $x$ to $y$ in $G_\infty$.
\end{definition}

Recall that $l^2_{k_n}$ is the $L^2$ space of functions $G_{k_n}\to\mathbb{R}$ with counting measure on the vertices. Using the canonical graph morphisms  $\tilde{\iota}_{k_n}:G_{k_n}\to G_\infty$ we identify each $l^2_{k_n}$ with the subspace of $l^2$ consisting of functions supported on  $\tilde{\iota}_{k_n}(G_{k_n})$.  It is obvious that if $x\in G_{k_n}$ is not a boundary point of $G_{k_n}$ then the neighbors of $x$ in $G_{k_n}$ are in one-to-one correspondence with the vertices neighboring $\tilde{\iota}_{k_n}(x)$ in $G_\infty$ and therefore
\begin{equation}\label{eq:LinftyisLkn}
	L_{\infty}f(\tilde{\iota}_{k_n}(x)) = L_{k_n} \bigl( f|_{\tilde{\iota}_{k_n}(G_{k_n})} \bigr)(x).
	\end{equation}

\subsection{Number of vertices of $G_n$}

It will be useful later to have an explicit expression for the number of vertices in $G_n$.  This may readily be computed from the decomposition in Figure~\ref{add}.

\begin{lemma}\label{csgetdegrees}
The number of vertices in $G_n$ is given by
\begin{equation*}
V_n = \dfrac{2^{2+n} +(-1)^{1+n}+9}{6}.
\end{equation*}
\end{lemma}

\begin{proof}
$G_n$ is constructed from a copy of $G_{n-1}$ and two copies of $G_{n-2}$ in which four boundary points are identified to a single vertex $u$, as shown for the case $n=3$ in Figure~\ref{add}.  Thus $V_n$ must satisfy the recursion $V_n = V_{n-1}+ 2V_{n-2}  -3$ with $V_0=2$, $V_1=3$. The formula given matches these initial values and satisfies the recursion because
\begin{align*}
	\lefteqn{6(V_{n-1}+2V_{n-2}-3)}\quad& \\
	& =9+2\cdot 9+(-1)^n+2(-1)^{n-1}+2^{1+n}+2\cdot2^n-18\\
	&= 9+(-1)^{1+n}+2^{2+n}
	\end{align*}
so the result follows by induction.
\end{proof}

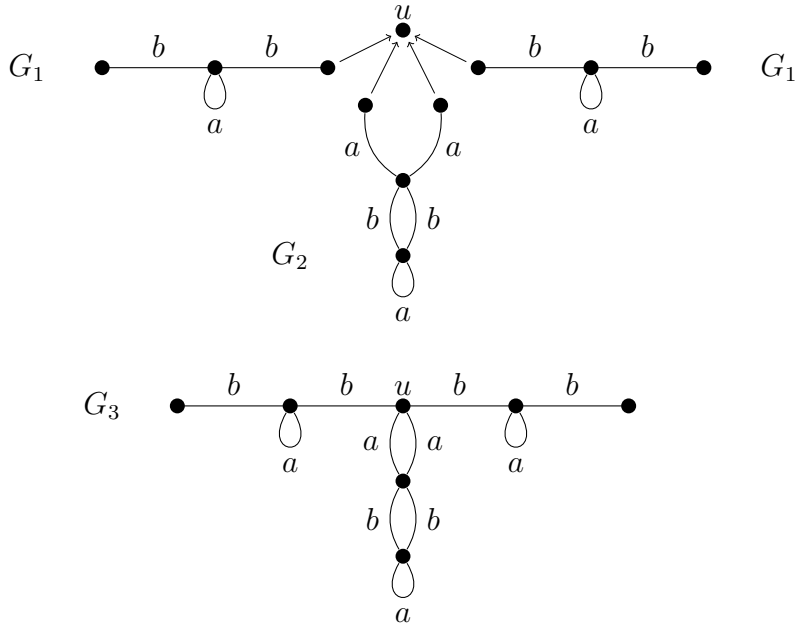
\begin{figure}[h!]
\centering
\begin{tikzpicture}
\path (7,5)node[circle,fill,inner sep=2pt](uu){} node[above]{$u$};
\draw   (2,4.5) node{$G_1$} (3,4.5) node[circle,fill,inner sep=2pt]{} --node[above]{$b$}  (4.5,4.5) node[circle,fill,inner sep=2pt](u){} --node[above]{$b$}  (6,4.5) node[circle,fill,inner sep=2pt](00b){};
\path[-, every loop/.style= {looseness=10, distance=20, in=300,out=240}]   (u) edge [loop below] node{$a$}();
\draw   (8,4.5) node[circle,fill,inner sep=2pt](00a){} --node[above]{$b$}  (9.5,4.5) node[circle,fill,inner sep=2pt](u){} --node[above]{$b$}  (11,4.5) node[circle,fill,inner sep=2pt]{} (12,4.5) node{$G_1$};
\path[-, every loop/.style= {looseness=10, distance=20, in=300,out=240}]   (u) edge [loop below] node{$a$}();
\path (5.5,2) node{$G_2$} (6.5,4) node[circle,fill,inner sep=2pt](200a){}   (7,3) node[circle,fill,inner sep=2pt](2u){}  (7.5,4) node[circle,fill,inner sep=2pt](200b){};
\draw[bend left] (2u) to node[left]{$a$} (200a);
\draw[bend right] (2u) to node[right]{$a$} (200b);
\path[-, every loop/.style= {looseness=10, distance=20, in=300,out=240}]   (2u)  --   (7,2) node[circle,fill,inner sep=2pt](211){} edge [loop below] node {$a$} (7,1) ;
\draw (2u) edge[bend right] node[left]{$b$} (211) edge[ bend left] node[right]{$b$} (211);
\draw[<-,very thin,shorten <=2pt, shorten >=2pt] (uu) to (00a);
\draw[<-,very thin,shorten <=2pt, shorten >=2pt] (uu) to (00b);
\draw[<-,very thin, shorten <=2pt, shorten >=2pt] (uu) to (200a);
\draw[<-,very thin, shorten <=2pt, shorten >=2pt] (uu) to (200b);
\draw  (3,0)node{$G_3$} (4,0) node[circle,fill,inner sep=2pt]{} --node[above]{$b$} (5.5,0) node[circle,fill,inner sep=2pt](300a){} --node[above]{$b$}  (7,0) node[circle,fill,inner sep=2pt](3u){} node[above]{$u$} --node[above]{$b$}  (8.5,0) node[circle,fill,inner sep=2pt](300b){} --node[above]{$b$} (10,0) node[circle,fill,inner sep=2pt]{};
\path[-, every loop/.style= {looseness=10, distance=20, in=300,out=240}]   (3u)  --   (7,-1) node[circle,fill,inner sep=2pt](311){} -- (7,-2) node[circle,fill,inner sep=2pt](3101){} edge [loop below] node {$a$}(7,-3)  ;
\path[every loop/.style= {looseness=10, distance=20, in=300,out=240}]   (300a) edge [loop below] node {$a$} ()  (300b) edge [loop below] node {$a$} ();
\draw (3u) edge[bend right] node[left]{$a$} (311) edge[ bend left] node[right]{$a$} (311);
\draw (311) edge[bend right] node[left]{$b$} (3101) edge[ bend left] node[right]{$b$} (3101);
\end{tikzpicture}\caption{$G_3$ constructed from a copy of $G_2$ and two of $G_1$.}
\label{add} 
\end{figure}

%%%%%%%%%%%%%%%%%%%%%%%%%%%%%%%%%%%%%%%%%%%%%%%%%%%%%%%%%%%%%%%%%%%%%%%%%%%%%%%%%%%%%%%%%%%%%%%%%%%%%%%%%%%%%%%%%%%%%%%%%%%%%%%%%%%%%%%%%%%%%%%%%%%%%%%%%%%%%%%%%%%%%%%%%%%%%%%%%%%%%%%%%%%%%%%%%%%%%%%%%%%%%%%%%%%%%%%%%%%%%%%%%%%%%%%%%%%%%%%%%%%%%%%%%%%%%%%%%%%%%%%%%%%%%%%%%%%%%%%%%%%%%%%%%%%%%%%%%%%%%%%%

\section{Dynamics for the spectrum of $G_n$}\label{section:dynamics}

It is well known that the spectra of Laplacians on self-similar graphs and fractals may often be described by using dynamical systems; we refer to~\cite{RammalToulouse,MalozemovTeplyaev,MR3558157} for typical examples and constructions of this type in both the physics and mathematics literature.  In particular, Grigorchuk and Zuk~\cite{MR1929716} gave a description of the Laplacian spectra for the graphs $\Gamma_n$ using a two-dimensional dynamical system.  Their method uses a self-similar group version of the Schur-complement (or Dirichlet-Neumann map) approach.  One might describe this  approach as performing a reduction at small scales, in that a single step of the dynamical system replaces many small pieces of the graph by equivalent weighted graphs.  In the case of $\Gamma_n$ one might think of decomposing it into copies of $G_2$ and $G_1$ and then performing an operation that reduces the former to weighted copies of $G_1$ and the latter to weighted copies of $G_0$, thus reducing $\Gamma_n$ to a weighted version of $\Gamma_{n-1}$.  The result is a dynamical system in which the characteristic polynomial of a weighted version of $\Gamma_n$ is written as the characteristic polynomial of a weighted version of $\Gamma_{n-1}$, composed with the dynamics that alters the weights. The spectrum is then found as the intersection of the Julia set of the dynamical system with a constraint on the weights.  See~\cite{MR1929716} for details and~\cite{MR2459869} for a similar method applied in different circumstances.

The approach we take here is different: we decompose at the macroscopic rather than the microscopic scale, splitting $G_n$ into a copy of $G_{n-1}$ and two of $G_{n-2}$, and then reasoning about the resulting relations between the characteristic polynomials.  The result is that our dynamical map is applied to the characteristic polynomials rather than appearing within a characteristic polynomial.  It is not a better method than that of~\cite{MR1929716} -- indeed it seems it may be more complicated to work with -- but it gives some insights that may not be as readily available from the more standard approach.

\subsection{Characteristic Polynomials}\label{subseccharpoly}
Our approach to analyzing the Laplacian spectrum for $G_n$ relies on the decomposition of $G_n$ into a copy of $G_{n-1}$ and two copies of $G_{n-2}$ as in Figure~\ref{add}. 

The following elementary lemma relates the characteristic polynomials of matrices under a decomposition of this type.  (This lemma  is a classical type and is presumably well known, though we do not know whether this specific formulation appears in the literature.)  It is written in terms of modifications of the Laplacian $L_n$ on certain subsets of $G_n$.   Consider a graph $G$ and a matrix $L$  indexed by the vertices of $G$ and such that the $jk$ entry is zero if there is no edge between the $j$ and $k$ vertices of $G$.  For $Z\subset G$ let us write $L^Z$ for the matrix with domain $\mathbb{R}^{G\setminus Z}$ and boundary condition $f|_Z=0$. The best-known cases are when $L$ is the graph Laplacian: then if $Z=\partial G$ we see $L^Z$ is the Dirichlet Laplacian and when $Z$ is empty $L^Z$ is the Neumann Laplacian. Also note that the characteristic polynomial of $L^Z$ is simply that of the matrix obtained from $L$ by deleting the rows and columns corresponding to the set $Z$.

\begin{lemma}\label{lem:charpoly}
Let $G$ be a finite graph, $u$ a fixed vertex, and $C(u)$ the set of simple cycles in $G$ containing $u$. Suppose $L$ is a matrix indexed by the vertices of $G$ with diagonal entries $d_j$  and off-diagonal entries $-c_{jk}$ such that $c_{jk}=0$ unless the $j$ and $k$ vertices of the graph are connected by an edge. If $D(\cdot)$ denotes the operation of taking the characteristic polynomial then
\begin{equation*}%\label{eqn:detwithoutu}
D(L)(\lambda) = (\lambda -d_u) D(L^{\{u\}})(\lambda) - \sum_{v\sim u} c^2_{uv}D(L^{\{u,v\}})(\lambda) + 2\sum_{Z \in C(u)} (-1)^{n(Z)-1}\pi(Z)  D(L^Z)(\lambda),
\end{equation*}
where $n(Z)$ is the number of vertices in $Z$ and $\pi(Z)$ is the product of the edge weights $c_{jk}$ along $Z$.
\end{lemma}
\begin{proof}
Recall that the determinant of a matrix $M=[m_{jk}]$ may be written as a sum over all permutations of the vertices of $G$ as follows: $\det(M)=\sum_\sigma \sgn(\sigma) \prod_j m_{j\sigma(j)}$.  Observe that each product term is non-zero only when the permutation $\sigma$ moves vertices along cycles on the graph and factor such $\sigma$ as $\sigma=\sigma'\sigma''$, where $\sigma'$ is the permutation on the $\sigma$ orbit of $u$ which we denote by $Z_\sigma$.  Take $M=\lambda-L$.  Using the Kronecker symbol $\delta_{jk}$ and writing $Z_\sigma^c$ for the complement of $Z_\sigma$ we write $D(L)$ as
\begin{equation*}
	\sum_{\sigma'} \sgn(\sigma') \prod_{j\in Z_\sigma} \bigl((\lambda-d_j) \delta_{j\sigma'(j)}+c_{j\sigma'(j)}\bigr)  \sum_{\sigma''}\sgn(\sigma'') \prod_{j\in Z_\sigma^c} \bigl((\lambda-d_j) \delta_{j\sigma''(j)}+c_{j\sigma''(j)}\bigr).
	\end{equation*}

For terms with $\sigma(u)=u$ the values of $\sigma''$ run over all permutations of the other vertices, so the corresponding term in the determinant sum is the product $(\lambda-d_u) D(L^{\{u\}})$.  When $\sigma'$ is a transposition $u\mapsto v\mapsto u$ we have $\sgn(\sigma')=-1$ and the product along $Z_\sigma$ is simply $c_{uv}^2$, so the corresponding terms have the form $-c^2_{uv}D(L^{\{u,v\}})$.

The remaining possibility is that the orbit of $u$ is a simple cycle $Z$ containing $n(Z)$ vertices. There are then two permutations $\sigma'$ that give rise to $Z$; these correspond to the two directions in which the vertices may be moved one position along $Z$. Each has $\sgn(\sigma')=(-1)^{n(Z)-1}$, so the corresponding terms in the determinant expansion are as follows
\begin{align*} 
	&\sum_{\sigma'} \sgn(\sigma') \prod_{j\in Z}c_{j\sigma(j)}  \sum_{\sigma''} \sgn(\sigma'') \prod_{j\in Z^c} \bigl((\lambda-d_j) \delta_{j\sigma(j)}+c_{j\sigma(j)}\bigr)\\
	&\quad =\sum_{\sigma'} (-1)^{n(Z)-1}\pi(Z) D(L^Z)\\
	&\quad = 2 (-1)^{n(Z)-1} \pi(Z) D(L^Z).
	\end{align*}
Combining these terms gives the desired expression for $D(L)(\lambda)$.
\end{proof}

In our application of this lemma we will consider graphs $A_n,B_n,C_n,D_n,E_n$ which are derived from the graphs $G_n$ discussed in the previous section.  We put $A_n=G_n$, $B_n$ to be $G_n$ with one boundary point deleted, $C_n$ to be $G_n$ with both boundary points deleted, $D_n$ to be $G_n$ with both boundary points deleted and also one vertex neighboring a boundary point deleted, and finally $E_n$ to be $G_n$ with one boundary point and its neighbor deleted.  The graphs $A_3$,$B_3$, and $C_3$ are shown in Figure~\ref{graphsABC}, while $D_3$ and $E_3$ are in Figure~\ref{graphsDE}.  It will be convenient to write $a_n(\lambda)$, $b_n(\lambda)$, $c_n(\lambda)$ for the characteristic polynomials of $A_n$, $B_n$ and $C_n$.  Note that then the roots of $a_n(\lambda)$ are the eigenvalues of the Neumann Laplacian and the roots of $c_n(\lambda)$ are the eigenvalues of the Dirichlet Laplacian on $G_n$.  Our initial goal is to describe these polynomials by using a dynamical system constructed from the decomposition in Figure~\ref{add}.

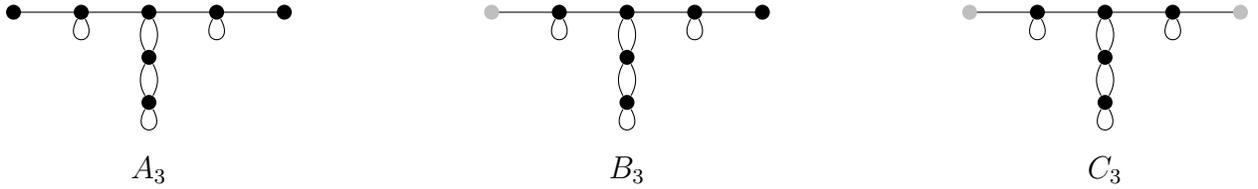
\begin{figure}[h!]
\centering
\begin{tikzpicture}[scale=0.6]
\draw  (4.5,3) node[circle,fill,inner sep=2pt]{} -- (6,3) node[circle,fill,inner sep=2pt](300a){} --  (7.5,3) node[circle,fill,inner sep=2pt](3u){} --  (9,3) node[circle,fill,inner sep=2pt](300b){} -- (10.5,3) node[circle,fill,inner sep=2pt]{};
\path[-, every loop/.style= {looseness=10, distance=20, in=300,out=240}]   (3u)  --   (7.5,2) node[circle,fill,inner sep=2pt](311){} -- (7.5,1) node[circle,fill,inner sep=2pt](3101){} edge [loop below](7.5,0)  ;
\path[every loop/.style= {looseness=10, distance=20, in=300,out=240}]   (300a) edge [loop below]  () (300b) edge [loop below]();
\draw (3u) edge[bend right]  (311) edge[ bend left]  (311);
\draw (311) edge[bend right]  (3101) edge[ bend left]  (3101);
\node at (7.5,-.5){$A_3$};
\end{tikzpicture}\hfill
\begin{tikzpicture}[scale=0.6]
\draw  (4.5,3) node[circle,fill=lightgray,inner sep=2pt]{} -- (6,3) node[circle,fill,inner sep=2pt](300a){} --  (7.5,3) node[circle,fill,inner sep=2pt](3u){} --  (9,3) node[circle,fill,inner sep=2pt](300b){} -- (10.5,3) node[circle,fill,inner sep=2pt]{};
\path[-, every loop/.style= {looseness=10, distance=20, in=300,out=240}]   (3u)  --   (7.5,2) node[circle,fill,inner sep=2pt](311){} -- (7.5,1) node[circle,fill,inner sep=2pt](3101){} edge [loop below](7.5,0)  ;
\path[every loop/.style= {looseness=10, distance=20, in=300,out=240}]   (300a) edge [loop below]  () (300b) edge [loop below]();
\draw (3u) edge[bend right]  (311) edge[ bend left]  (311);
\draw (311) edge[bend right]  (3101) edge[ bend left]  (3101);
\node at (7.5,-.5){$B_3$};
\end{tikzpicture}\hfill
\begin{tikzpicture}[scale=0.6]
\draw  (4.5,3) node[circle,fill=lightgray,inner sep=2pt]{} -- (6,3) node[circle,fill,inner sep=2pt](300a){} --  (7.5,3) node[circle,fill,inner sep=2pt](3u){} --  (9,3) node[circle,fill,inner sep=2pt](300b){} -- (10.5,3) node[circle,fill=lightgray,inner sep=2pt]{};
\path[-, every loop/.style= {looseness=10, distance=20, in=300,out=240}]   (3u)  --   (7.5,2) node[circle,fill,inner sep=2pt](311){} -- (7.5,1) node[circle,fill,inner sep=2pt](3101){} edge [loop below](7.5,0)  ;
\path[every loop/.style= {looseness=10, distance=20, in=300,out=240}]   (300a) edge [loop below]  () (300b) edge [loop below]();
\draw (3u) edge[bend right]  (311) edge[ bend left]  (311);
\draw (311) edge[bend right]  (3101) edge[ bend left]  (3101);
\node at (7.5,-.5){$C_3$};
\end{tikzpicture}
\caption{Graphs $A_3$, $B_3$, $C_3$. Rows and columns corresponding to grey vertices are deleted in the corresponding matrices.}%\label{sixgraphs}
\label{graphsABC}
\end{figure}

\begin{proposition}\label{prop:polydynamics}
For $n\geq4$ the characteristic polynomials $a_n$, $b_n$ and $c_n$ of the graphs $A_n$, $B_n$ and $C_n$ satisfy
\begin{align*}
a_n &= \bigl(2b_{n-1}-3\lambda c_{n-1}-2g_{n-1}\bigr) b_{n-2}^2 + 2a_{n-2}b_{n-2}c_{n-1},\\
b_n &=  \bigl(2b_{n-1}-3\lambda c_{n-1}-2g_{n-1}\bigr) b_{n-2}c_{n-2} + (a_{n-2}c_{n-2}+b_{n-2}^2)c_{n-1},\\
c_n &= \bigl(2b_{n-1}-3\lambda c_{n-1}-2g_{n-1}\bigr) c_{n-2}^2 + 2b_{n-2}c_{n-2}c_{n-1},
\end{align*}
where 
\begin{equation}\label{eqn:defnofgn}
	g_{n-1}=\prod_{1\leq j <\frac n2}\bigl( c_{n-2j}\bigr)^{2^{j-1}}.
	\end{equation}
\end{proposition}
\begin{proof}
Figure~\ref{add} illustrates the fact that $G_n$ can be obtained from one copy of $G_{n-1}$ and two copies of $G_{n-2}$ by identifying the two boundary vertices of $G_{n-1}$ and one boundary vertex from each copy of $G_{n-2}$ into a single vertex which we denote by $u$. We apply Lemma~\ref{lem:charpoly} to $L_n$ on $G_n$ with vertex $u$ to compute the characteristic polynomial.  This involves modifying the Laplacian matrix on various sets of vertices.  The subgraphs with modified vertices are $A_n$, $B_n$, and $C_n$ as in Figure~\ref{graphsABC} and also $D_n$, $E_n$ as in Figure~\ref{graphsDE}.

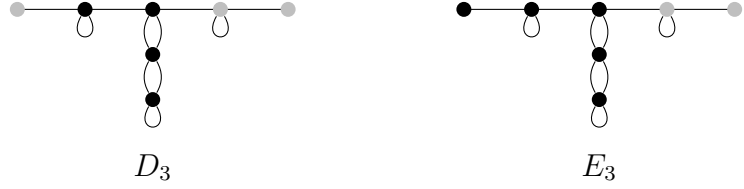
\begin{figure}
\centering
\begin{tikzpicture}[scale=0.6]
\draw  (4.5,3) node[circle,fill=lightgray,inner sep=2pt]{} -- (6,3) node[circle,fill,inner sep=2pt](300a){} --  (7.5,3) node[circle,fill,inner sep=2pt](3u){} --  (9,3) node[circle,fill=lightgray,inner sep=2pt](300b){} -- (10.5,3) node[circle,fill=lightgray,inner sep=2pt]{};
\path[-, every loop/.style= {looseness=10, distance=20, in=300,out=240}]   (3u)  --   (7.5,2) node[circle,fill,inner sep=2pt](311){} -- (7.5,1) node[circle,fill,inner sep=2pt](3101){} edge [loop below](7.5,0)  ;
\path[every loop/.style= {looseness=10, distance=20, in=300,out=240}]   (300a) edge [loop below]  () (300b) edge [loop below]();
\draw (3u) edge[bend right]  (311) edge[ bend left]  (311);
\draw (311) edge[bend right]  (3101) edge[ bend left]  (3101);
\node at (7.5,-.5){$D_3$};
\end{tikzpicture}\hspace{2cm}
\begin{tikzpicture}[scale=0.6]
\draw  (4.5,3) node[circle,fill,inner sep=2pt]{} -- (6,3) node[circle,fill,inner sep=2pt](300a){} --  (7.5,3) node[circle,fill,inner sep=2pt](3u){} --  (9,3) node[circle,fill=lightgray,inner sep=2pt](300b){} -- (10.5,3) node[circle,fill=lightgray,inner sep=2pt]{};
\path[-, every loop/.style= {looseness=10, distance=20, in=300,out=240}]   (3u)  --   (7.5,2) node[circle,fill,inner sep=2pt](311){} -- (7.5,1) node[circle,fill,inner sep=2pt](3101){} edge [loop below](7.5,0)  ;
\path[every loop/.style= {looseness=10, distance=20, in=300,out=240}]   (300a) edge [loop below]  () (300b) edge [loop below]();
\draw (3u) edge[bend right]  (311) edge[ bend left]  (311);
\draw (311) edge[bend right]  (3101) edge[ bend left]  (3101);
\node at (7.5,-.5){$E_3$};
\end{tikzpicture}
\caption{Graphs $D_3$ and $E_3$. Rows and columns corresponding to shaded vertices are deleted in the corresponding matrices.}
\label{graphsDE}
\end{figure}

For $n\geq4$ the point $u$ has one neighbor in each copy of $G_{n-2}$ as well as two neighbors in the copy of $G_{n-1}$ that lie on a simple cycle which was formed by identifying the boundary vertices.  Accordingly the vertex modifications involved in applying Lemma~\ref{lem:charpoly} are as follows. 

Modifying $A_n$ at $u$ gives the disjoint union of two copies of $B_{n-2}$ and one of $C_{n-1}$.  To modify on $\{u,v\}$ observe that  if $v$ is on one of the two copies of $G_{n-2}$ then the result is one copy of each of $B_{n-2}$, $E_{n-2}$ and $C_{n-1}$, while if $v$ is on the copy of $G_{n-1}$ then we see two copies of $B_{n-2}$ and one of $D_{n-1}$. The most interesting modification is that for the cycle. Modifying at $u$ turns the two copies of $G_{n-2}$ into two copies of $B_{n-2}$. The rest of the cycle runs along the shortest path in $G_{n-1}$ between the boundary points that were identified at $u$.  Modifying along this causes $G_{n-1}$ to decompose into the disjoint union of one, central, copy of $C_{n-2}$, two copies of $C_{n-4}$ equally spaced on either side and, inductively, $2^{j-1}$ copies of $C_{n-2j}$ for each $j$ such that $2j<n$, equally spaced between those obtained at the previous step. There are also loops along this path which now have no vertices and therefore each have characteristic polynomial $1$. The characteristic polynomial of this collection of $C_{n-2j}$ graphs is $g_{n-1}$.

If we write $d_n$ and $e_n$ for the characteristic polynomials of $D_n$ and $E_n$ respectively, then from the above reasoning we conclude that
\begin{equation}\label{dynampolyeqn1}
	a_n = (\lambda-4)b_{n-2}^2 c_{n-1} - 2 b_{n-2} e_{n-2} c_{n-1} -2 b_{n-2}^2 d_{n-1} - 2 b_{n-2}^2 g_{n-1}.
\end{equation}
Similar arguments beginning with $B_n$ or $C_n$ instead of $A_n$ allow us to verify that
\begin{align}
b_n &= (\lambda-4)b_{n-2}c_{n-2} c_{n-1} -  b_{n-2}d_{n-2} c_{n-1} -  c_{n-2}e_{n-2} c_{n-1}\notag\\
 &\quad - 2 b_{n-2} c_{n-2} d_{n-1} - 2  b_{n-2}c_{n-2} g_{n-1},\label{dynampolyeqn2}\\
c_n &= (\lambda-4)c_{n-2}^2 c_{n-1} - 2 c_{n-2}d_{n-2} c_{n-1} - 2 c_{n-2}^2 d_{n-1} - 2  c_{n-2}^2 g_{n-1}. \notag
\end{align}

Another use of Lemma~\ref{lem:charpoly}  allows us to relate some of our modified graphs to one another by performing one additional vertex modification.  For example, for $n\geq3$ we get $C_n$ from $B_n$ by modifying at one boundary vertex, and this vertex does not lie on a cycle.  Deleting the corresponding neighbor gives $D_n$, so we  have $b_n= (\lambda-1)c_n -d_n$.  In a like manner we obtain 
$a_n=(\lambda-1)b_n-e_n$.  These can be used to eliminate $d_n$ and $e_n$ from equations~\eqref{dynampolyeqn1} and \eqref{dynampolyeqn2} and obtain the desired conclusion.
\end{proof}

The initial polynomials $a_n$, $b_n$, $c_n$ for the recursion in Proposition~\ref{prop:polydynamics} are those with $0\leq n\leq 3$.  They may be computed for $n=0,1$ directly from the Laplacians of the graphs in Figure~\ref{graphsABC}.
\begin{align}
	a_0&= \lambda(\lambda-2)& & b_0= \lambda-1& & c_0= 1\notag\\
	a_1&= \lambda(\lambda-1)(\lambda-3) & & b_1= \lambda^2-3\lambda+1 & & c_1=\lambda-2 \label{eqn:startpolys1}
\end{align}
For $n=2,3$ we can use a variant of the argument in the proof of Proposition~\ref{prop:polydynamics}, taking the initial graph and modifying the connecting vertex $u$ by using Lemma~\ref{lem:charpoly}. In these cases there is no simple cycle, so we need only consider the self-interaction term and the terms corresponding to neighbors, of which there are three: one in the copy of $G_{n-1}$ which is connected by a double edge, so $c^2_{uv}=4$, and one in each of the copies of $G_{n-2}$.

For $A_2$ modifying $u$ gives a copy of $C_1$ and two of $B_0$.  Additionally modifying a neighbor in one of the two $G_{0}$ copies produces a $C_0$, a $B_0$ and a $C_1$, while deleting the neighbor in the copy of $G_1$ decomposes the whole graph into two $B_0$ copies and three $C_0$ copies. Since $c_0=1$ we suppress it in what follows. From this we have an equation for $a_2$.  Similar reasoning, noting that $u$ has fewer neighbors in $B_2$ and $C_2$, gives results for $b_2$ and $c_2$.  We summarize them as
\begin{align}
	a_2&= (\lambda-4)b_0^2 c_1 - 2b_0c_1 - 4b_0^2 = \lambda(\lambda^3-8\lambda^2+15\lambda-8), \notag\\
	b_2&= (\lambda-4)b_0 c_1 -  c_1  -4b_0= \lambda^3 -7\lambda^2+9\lambda-2, \label{eqn:startpolys2}\\
	c_2&= (\lambda-4)c_1  - 4 =\lambda^2 -6\lambda+4. \notag
	\end{align}

For $A_3$ things are more like they were in Proposition~\ref{prop:polydynamics}. Modifying at $u$ gives $C_2$ and two copies of $B_1$, additionally modifying at a neighbor in the $G_1$ copies gives a $C_2$, $B_1$ and $D_1$, but $D_1=B_0$. Modifying at $u$ and the neighbor in the $G_2$ copy gives a $C_1$ and two copies of $B_1$. Reasoning in the same manner for $B_3$ and $C_3$ we have
\begin{align}
	a_3&= (\lambda-4)b_1^2c_2 - 2 b_1c_2d_1 -4 b_1^2c_1\notag \\
	%&=  (\lambda c_2-4c_1-4c_2) b_1^2 - 2 b_0 b_1c_2 \notag\\
	&= \lambda(\lambda-2)(\lambda^2-3\lambda+1)(\lambda^3-11\lambda^2+31\lambda-14), \notag\\
	b_3&=(\lambda-4)b_1c_1c_2 - b_1c_2 - b_0c_1c_2 - 4b_1c_1^2 \label{eqn:startpolys3} \\ 
	%&= (\lambda c_2-4c_1-4c_2)b_1c_1 -(b_1+b_0c_1) c_2 \notag\\
	&= \lambda^6 - 15\lambda^5+ 79\lambda^4 -182\lambda^3 +181\lambda^2 -62\lambda+4,\notag\\
	c_3&= (\lambda-4)c_1^2 c_2 - 2c_1 c_2 - 4 c_1^3 \notag \\%=(\lambda c_2-4c_1-4c_2)c_1^2 - 2c_1c_2 
		&= (\lambda-2) (\lambda^4-12\lambda^3+42\lambda^2-44\lambda+8). \notag
	\end{align}

\begin{proposition}\label{prop:cdynamics}
The characteristic polynomials $a_n$, $b_n$ and $c_n$ may be obtained from the initial data~\eqref{eqn:startpolys1},\eqref{eqn:startpolys2},\eqref{eqn:startpolys3} by using the following recursions, where we note that the recursion for $c_n$ involves only  $c$ terms (because the $g_n$ are products of $c_k$ terms, see~\eqref{eqn:defnofgn}), that for $b_n$ involves only $b$ and $c$ terms, and that for $a_n$ involves all three sequences. 
\begin{gather}
	\frac{c_n}{c_{n-2}}  =   \Bigl( \frac{c_{n-1}}{ c_{n-3}}\Bigr)^2 + 2  c_{n-1}g_{n-2}   -4c_{n-2} g_{n-1},\quad n\geq3,\label{eqn:cdynamics1}\\
	b_{2m}= c_{2m} \Bigl( b_0 - \sum_1^m \frac{g_{2j}}{c_{2j}} \Bigr), \qquad b_{2m+1}= c_{2m+1} \Bigl( \frac{b_1}{c_1} - \sum_1^m \frac{g_{2j+1}}{c_{2j+1}}\Bigr), \quad m\geq1, \label{eqn:cdynamics2}\\
	a_{n}c_n = b_n^2-g_n^2, \quad n\geq0. \label{eqn:cdynamics3}
	\end{gather}
\end{proposition}
\begin{proof}
Multiplying the $a_n$ equation in Proposition~\ref{prop:polydynamics} by $c_{n-2}^2$, the $b_n$ one by $-2b_{n-2}c_{n-2}$ and the $c_n$ one by $b_{n-2}^2$ and summing the results gives the following relationship for $n\geq 4$:
\begin{equation*}
	a_n c_{n-2}^2 -2b_n b_{n-2}c_{n-2} +c_nb_{n-2}^2 = 0,
	\end{equation*}
which can also be verified for $n=2,3$ from~\eqref{eqn:startpolys1},\eqref{eqn:startpolys2}, and~\eqref{eqn:startpolys3}. We use it to eliminate $a_{n-2}$ from the equation for $b_n$ and thereby obtain recursions for $b_n$ and $c_n$ that do not involve the sequence $a_n$. It is convenient to do so by computing (in the case that $c_{n-2}\neq0$)
\begin{align}
	a_n c_n - b_n^2  &= \frac{1}{c_{n-2}^2} \bigl( 2b_n c_n b_{n-2}c_{n-2} - c_n^2 b_{n-2}^2 - b_n^2c_{n-2}^2 \bigr) \notag\\
				&=  \frac{-(b_nc_{n-2} - b_{n-2}c_n)^2 }{c_{n-2}^2} \quad \text{when $n\geq2$},\label{dynamicpolyeqn4}
	\end{align}
because we may now compute from Proposition~\ref{prop:polydynamics} and apply~\eqref{dynamicpolyeqn4} with $n$ replaced by $n-2$ to obtain for $n\geq4$
\begin{align*}
	b_n c_{n-2} - c_n b_{n-2}
	&= c_{n-1} \bigl( a_{n-2}c_{n-2}^2 + b_{n-2}^2 c_{n-2} -2b_{n-2}^2c_{n-2} \bigr)\\
	&= c_{n-2} c_{n-1} \bigl( a_{n-2}c_{n-2} - b_{n-2}^2  \bigr)\\
	&= \frac{-c_{n-2}c_{n-1}(b_{n-2}c_{n-4}-b_{n-4}c_{n-2})^2}{c_{n-4}^2}.
	\end{align*}
We can use this to get, for $n\geq4$,
\begin{equation*}
	b_n - \frac{c_n}{c_{n-2}} b_{n-2}
	= -c_{n-1}c_{n-3}^2 c_{n-5}^4\dotsm 
	\begin{cases}
		c_3^{2^{(n-4)/2}} ( b_2- c_2b_0)^{2^{(n-2)/2}} &\text{ if $n$ is even,}\\
		c_4^{2^{(n-5)/2}} ( b_3- c_3b_1/c_1)^{2^{(n-3)/2}} &\text{ if $n$ is odd,}
		\end{cases}
	\end{equation*}
however one may compute directly from~\eqref{eqn:startpolys1}, \eqref{eqn:startpolys2} and~\eqref{eqn:startpolys3} that $ b_2- c_2b_0=-c_1$ and $b_3-c_3b_1/c_1=-c_2$, so that for $n\geq2$
\begin{equation}\label{eqn:bn}
	b_n - \frac{c_n}{c_{n-2}} b_{n-2} = - g_n,
	\end{equation}
from which we obtain the expressions in~\eqref{eqn:cdynamics2} by summation and~\eqref{eqn:cdynamics3} by substitution into~\eqref{dynamicpolyeqn4}.  We also have~\eqref{eqn:cdynamics3} for $n=0,1$ by~\eqref{eqn:startpolys1} and $g_0=g_1=1$.

We may also use this to eliminate $b_n$ from the expression for $c_n$ in Proposition~\ref{prop:polydynamics}.  A convenient way to do so is to rewrite the equation for $c_n$ as
\begin{equation}\label{eqn:cnfrombn}
	\frac{c_n}{c_{n-1}c_{n-2}^2} = 2\Bigl( \frac{b_{n-1}}{c_{n-1}} + \frac{b_{n-2}}{c_{n-2}} -\frac{g_{n-1}}{c_{n-1}} \Bigr) -3\lambda,
	\end{equation}
which holds for $n\geq4$ and can be checked for $n=2,3$ from ~\eqref{eqn:startpolys1}, \eqref{eqn:startpolys2} and~\eqref{eqn:startpolys3}, and use~\eqref{eqn:bn} to eliminate the $b_{n-1}/c_{n-1}$ term.  Comparing the result with~\eqref{eqn:cnfrombn} for the case $n-1$ we have, for $n\geq3$, both
\begin{align*}
	\frac{c_n}{c_{n-1}c_{n-2}^2} &= 2\Bigl(  \frac{b_{n-2}}{c_{n-2}} +\frac{b_{n-3}}{c_{n-3}}  -\frac{2g_{n-1}}{c_{n-1}} \Bigr) -3\lambda,,\\
	\frac{c_{n-1}}{c_{n-2} c_{n-3}^2} &= 2\Bigl( \frac{b_{n-2}}{c_{n-2}} + \frac{b_{n-3}}{c_{n-3}} -\frac{g_{n-2}}{c_{n-2}} \Bigr) -3\lambda,
	\end{align*}
the difference of which is
\begin{equation*}
	\frac{c_n}{c_{n-1}c_{n-2}^2} - \frac{c_{n-1}}{c_{n-2} c_{n-3}^2}
	= 2  \frac{g_{n-2}}{c_{n-2}}   -4\frac{g_{n-1}}{c_{n-1}}
	\end{equation*}
and may be rearranged to give~\eqref{eqn:cdynamics1}.
\end{proof}

\subsection{Localized Eigenfunctions and factorization of characteristic polynomials.}\label{subsec:locefns}

In this section we consider the spectrum of the Dirichlet Laplacian on $G_n$, for which the characteristic polynomial is $c_n$.  
We define $\gamma_0=c_0=1$ and  recursively take  $\gamma_n$ to have no roots in common with $\gamma_k$ for $k<n$ and such that
\begin{equation}\label{eqn:recursionforcn}
	c_n=\gamma_n\prod_{k=1}^{n-1} \gamma_k^{s_{n,k}}.
\end{equation}
for some indices $s_{n,k}\geq0$.  The main goal of the section is to give a recursive formula for the indices $s_{n,k}$; this is achieved in Theorem~\ref{thm:factorizationofcn} as a consequence of a description of certain eigenfunctions in Theorem~\ref{thm:structure of DN efns}.  A key feature of this description is the construction of eigenfunctions that satisfy both Dirichlet and Neumann boundary conditions, which we label DN-eigenfunctions.

We now fix an integer $m\geq1$ and a root $\lambda$ of $\gamma_m$.  The proofs of the preceding theorems require us to study the solutions of $Lf=\lambda f$ on $G_n\setminus \partial G_n$ for $n>m$. Throughout the section $f$ will refer to such a solution, though $n$ will change.  Although our main argument is an induction on $n$, the first few cases $n=m,m+1,m+2$ are a little different than the others, so are done in separate lemmas.  

We need a small amount more notation in order to proceed. Since $\lambda$ is not a root of $\gamma_k$ for $k<n$ it is not a Dirichlet eigenvalue of the Laplacian on $G_k$, $k<n$.  In particular, there is a unique solution $h_k$ to the boundary value problem $Lh_k=\lambda h_k$ on $G_k\setminus\partial G_k$ with data $1$ at one boundary point and $0$ at the other boundary point.  We will refer to these functions in our diagrams of solutions on $G_k$ for larger $k$ below.   We will also need notation for the Laplacian of these functions at the endpoints, which only involves the edge difference at the boundary point. We will call this the Neumann derivative.  By a slight abuse of notation we denote the Neumann derivative of $h_k$ at the boundary point where $h_k=0$ by $\partial h_k(0)$ and similarly that at the boundary point where $h_k=1$ by $\partial h_k(1)$. The usefulness of these is that when the boundary points of copies of $G_{n-2}$ and $G_{n-1}$ are identified to produce $G_n$ we obtain the Laplacian at the gluing point by summing the Neumann derivatives at the points that were glued. This fact will be used without further comment.

The final thing for which we need notation is a symmetry of $G_n$. Recall that $G_n$ is constructed from two copies of $G_{n-2}$, labelled ``left'' and ``right'', and one copy of $G_{n-1}$ with identification of boundary points to a single gluing point $u$ as in Figure~\ref{add}.  We define a graph isomorphism $\Phi_n$ on $G_n$ to swap the labels on the copies of $G_{n-2}$ (here it is assumed this fixes the gluing point) and to restrict to give the map $\Phi_{n-1}$ on the copy of $G_{n-1}$.

\begin{lemma}\label{lem:adjpointszeroimpliesvanishes}
Suppose $Lf=\lambda f$ on $G_n\setminus\partial G_n$ for some $n\leq m$. If $f$ vanishes at two adjacent points on the shortest path between the boundary points then it is identically zero.  In particular, $\partial h_k(0)$ is non-zero for $k\leq m$.
\end{lemma}
\begin{proof}
The shortest path between boundary points is an interval containing vertices at which copies of $G_k$, $k<m$ are attached.  Take the two vertices at which $f$ vanishes, and a point adjacent to one of them on the path, and label these in order as $x$, $y$, $z$ with $f(x)=f(y)=0$. Since $\lambda$ is not a Dirichlet eigenvalue for $G_k$, $f$ must vanish identically on any graph attached at $y$ and thus $0=\lambda f(y)=Lf(y)=2f(y)-f(x)-f(z)=-f(z)$.  The fact that $f$ vanishes at any neighbor of two adjacent zeros of $f$ implies $f\equiv0$ on the shortest path by connectedness, thus on all attached graphs as already mentioned, and therefore on $G_n$.  For the last statement, if $\partial h_k(0)=0$ then $h_k$ vanishes at the boundary point and its neighbor, so is identically zero in contradiction to the fact that it is $1$ at the other boundary point.
\end{proof}

\begin{figure}
\centering
\begin{tikzpicture}
[vertex/.style={circle,fill,inner sep=2pt},node distance=1.5cm]
% First graph
\node[vertex] (bdychain11) at (0,0) [label=above:$$] {};
\node[vertex] (bdychain12) [right=of bdychain11, label=above:$u$]{};
\node[vertex] (bdychain13) [right=of bdychain12] [label=above:$$] {};
\draw (bdychain11)-- node[below]{$G_{m-2}$} (bdychain12) -- node[below]{$G_{m-2}$} (bdychain13);
\path[-, every loop/.style= {looseness=10, distance=40, in=300,out=240}]   (bdychain12) edge [loop below] node{$G_{m-1}$}();
% Second graph
\node[vertex] (bdychain21) [right=of bdychain13, label=above:$0$]{};
\node[vertex] (bdychain22) [right=of bdychain21, label=above:$1$]{};
\node[vertex] (bdychain23) [right=of bdychain22] [label=above:$0$] {};
\draw (bdychain21)-- node[below]{$h_{m-2}$} (bdychain22) -- node[below]{$h_{m-2}$} (bdychain23);
\path[-, every loop/.style= {looseness=10, distance=40, in=300,out=240}]   (bdychain22) edge [loop below] node{}();
% Third graph
\node[vertex] (bdychain31)  [right=of bdychain23, label=above:$1$]{};
\node[vertex] (bdychain32) [right=of bdychain31, label=above:$0$]{};
\node[vertex] (bdychain33) [right=of bdychain32] [label=above:$-1$] {};
\draw (bdychain31)-- node[below]{$h_{m-2}$} (bdychain32) -- node[below]{$-h_{m-2}$} (bdychain33);
\path[-, every loop/.style= {looseness=10, distance=40, in=300,out=240}]   (bdychain32) edge [loop below] node{}();
\end{tikzpicture}
\caption{Decomposition of $G_m$ (left); Dirichlet eigenfunction $h$ (center); Antisymmetric solution of $L f=\lambda f$ on $G_m\setminus\partial G_m$ (right)}
\label{fig:efnsGm}
\end{figure}

\begin{proposition}\label{prop:efnsGm}
The eigenvalue $\lambda$ is simple. We take as a basis element the eigenfunction normalized to have value $1$ at the gluing point $u$. We denote the eigenfunction by $h$ and depict in the center of Figure~\ref{fig:efnsGm}. It is symmetric and has non-zero Neumann derivative $\partial h_{m-2}(0)$ at both boundary points.  There is one other solution to $Lf=\lambda f$ on $G_m\setminus\partial G_m$, which is depicted on the right of Figure~\ref{fig:efnsGm}. It has Neumann derivatives $\pm\partial h_{m-2}(1)$.
\end{proposition}

\begin{proof}
Since $\lambda$ is a root of $\gamma_m$ there is an eigenfunction on $G_m$.  Its value at the gluing point $u$ determines the function uniquely on $G_m$ because it and the values on $\partial G_m$ serve as boundary data on the copies of $G_{m-2}$ and $G_{m-1}$ in $G_m$ and $\lambda$ is not a Dirichlet eigenvalue for these graphs.  This shows the eigenspace is one-dimensional and allows us to normalize to get basis element $h$ with $h(u)=1$ as in the center of Figure~\ref{fig:efnsGm}.  The boundary data is $\Phi_m$-symmetric so  $h$ is $\Phi_m$-symmetric. It is apparent from the diagram that its Neumann derivative is  $\partial h_{m-2}(0)$, and this is non-zero by Lemma~\ref{lem:adjpointszeroimpliesvanishes}.

To see that the antisymmetric function depicted on the right of  Figure~\ref{fig:efnsGm} is a solution of $Lf=\lambda f$ on $G_m\setminus\partial G_m$ we need only check the equation holds at the gluing point. The function vanishes on the copy of $G_{m-1}$ because it is zero at the boundary points, both of which are at $u$, so there is no Neumann derivative from this subgraph. Antisymmetry ensures the Neumann derivatives from the copies of $G_{m-2}$ cancel at the gluing point, verifying $Lf(u)=\lambda f(u)=0$ there.

It remains to see that there are no other solutions of $Lf=\lambda f$ on $G_m\setminus\partial G_m$. Any such $f$ could be assumed $\Phi_m$-symmetric by subtracting a copy of the antisymmetric solution and to have $f(u)=0$ by subtracting a copy of $h$.  But then it would be identically zero on the copy of $G_{m-1}$ and equal to symmetrically arranged copies of $h_{m-2}$ on the copies of $G_{m-2}$. The sum of the Neumann derivatives at $u$ would then be a non-zero multiple of $\partial h_{m-2}(0)\neq0$ in contradiction to $Lf(u)=\lambda f(u)$, so there is no such solution.
\end{proof}

\begin{corollary}\label{cor:sumofNeumderivs}
$2(\partial h_{m-2}(1)+\partial h_{m-1}(0)+\partial h_{m-1}(1))=\lambda$.
\end{corollary}
\begin{proof}
For the eigenfunction $h$ in the proposition we have $\lambda=\lambda h(u)=Lh(u)$ is the sum of the Neumann derivatives from the subgraphs glued at $u$. Two are the copies of $G_{m-2}$ which each provide Neumann derivative $\partial h_{m-2}(1)$.  The other is the copy of $G_{m-1}$ with both boundary values equal to $1$. It is apparent that this function is the sum of $h_{m-1}$ and a copy of $h_{m-1}$ reflected via $\Phi_{m-1}$, so each boundary point has Neumann derivative $\partial h_{m-1}(0)+\partial h_{m-1}(1)$. Summing two copies of this with the contributions from $G_{m-2}$ gives the formula.
\end{proof}

\begin{figure}
\centering
\begin{tikzpicture}
[vertex/.style={circle,fill,inner sep=2pt},node distance=1.5cm]
% First graph
\node[vertex] (bdychain11) at (0,0) [label=above:$$] {};
\node[vertex] (bdychain12) [right=of bdychain11, label=above:$u$]{};
\node[vertex] (bdychain13) [right=of bdychain12] [label=above:$$] {};
\draw (bdychain11)-- node[below]{$G_{m-1}$} (bdychain12) -- node[below]{$G_{m-1}$} (bdychain13);
\path[-, every loop/.style= {looseness=10, distance=40, in=300,out=240}]   (bdychain12) edge [loop below] node{$G_{m}$}();
% Second graph
\node[vertex] (bdychain21) [right=of bdychain13, label=above:$1$]{};
\node[vertex] (bdychain22) [right=of bdychain21, label=above:$0$]{};
\node[vertex] (bdychain23) [right=of bdychain22] [label=above:$0$] {};
\draw (bdychain21)-- node[below]{$h_{m-1}$} (bdychain22) -- node[below]{$0$} (bdychain23);
\path[-, every loop/.style= {looseness=10, distance=40, in=300,out=240}]   (bdychain22) edge [loop below] node{$\kappa h$}();
% Third graph
\node[vertex] (bdychain31)  [right=of bdychain23, label=above:$0$]{};
\node[vertex] (bdychain32) [right=of bdychain31, label=above:$0$]{};
\node[vertex] (bdychain33) [right=of bdychain32] [label=above:$1$] {};
\draw (bdychain31)-- node[below]{$0$} (bdychain32) -- node[below]{$h_{m-1}$} (bdychain33);
\path[-, every loop/.style= {looseness=10, distance=40, in=300,out=240}]   (bdychain32) edge [loop below] node{$\kappa h$}();
\end{tikzpicture}
\caption{Decomposition of $G_{m+1}$ and solutions of $L f=\lambda f$ on $G_{m+1}\setminus \partial G_{m+1}$.}
\label{fig:efnsGm+1}
\end{figure}

\begin{lemma}\label{lem:efnsGm+1}
On $G_{m+1}$ there are no eigenfunctions with eigenvalue $\lambda$. There is a unique constant $\kappa\neq0$ such that a basis for the solutions of $Lf=\lambda f$ on $G_{m+1}\setminus\partial G_{m+1}$ is as shown in Figure~\ref{fig:efnsGm+1}.  These basis elements satisfy both Dirichlet and Neumann boundary conditions at one boundary point and have Neumann derivative $\partial h_{m-1}(1)$ at the other boundary point.
\end{lemma}
\begin{proof}
We see that $f$ must satisfy the same equation on the copy of $G_m$ inside $G_{m+1}$ and has both boundary values equal to each other on this copy. From Proposition~\ref{prop:efnsGm} it is then a multiple of the eigenfunction $h$, so at the gluing point $f(u)=0$.  It follows that $f$ is determined entirely on the copies of $G_{m-1}$ by its data on $\partial G_{m+1}$, so the lemma is proved once we show there is $\kappa\neq0$ that makes the functions in Figure~\ref{fig:efnsGm+1} satisfy the equation, which is simply a matter of checking we can make $Lf(u)=\lambda f(u)=0$.

In the diagrams the Neumann derivative from one copy of $G_{m-1}$ is zero and from the other is $\partial h_{m-1}(0)$, which is non-zero by Lemma~\ref{lem:adjpointszeroimpliesvanishes}. If we glue the boundary points in a copy $G_m$ carrying the eigenfunction $h$, the resulting Neumann derivative is $2\partial h_{m-2}(0)\neq0$ by Proposition~\ref{prop:efnsGm}. Now $Lf(u)=0$ if and only if $2\kappa\partial h_{m-2}(0)=-\partial h_{m-1}(0)$, so $\kappa$ is unique and non-zero.
\end{proof}

The Dirichlet-Neumann boundary conditions at one boundary point of $G_{m+1}$ allow us to extend to any graph glued at that point while retaining the condition that $Lf=\lambda f$.  A useful consequence follows.

\begin{corollary}\label{cor:allefnsonGmodd}
On $G_n$ with $n>m$ and $n-m$ odd, each of the boundary points is also a boundary point for a copy of $G_{m+1}$. Setting $f$ to be the function in Figure~\ref{fig:efnsGm+1} on this copy of $G_{m+1}$ and $f\equiv0$ on the rest of $G_n$ defines a solution to $Lf=\lambda f$ on $G_n\setminus\partial G_n$.
\end{corollary}

\begin{lemma}\label{lem:efnsGm+2}
The solutions of  $Lf=\lambda f$ on $G_{m+2}\setminus\partial G_{m+2}$ are as shown in Figure~\ref{fig:efnsGm+2}.  If the boundary points are identified then the solution on the left has both Dirichlet and Neumann conditions at the identified point.
\end{lemma}

\begin{proof}
We first check that the two functions shown are solutions to the equation, which only requires that we verify $Lf(u)=\lambda f(u)$ at the gluing point $u$.  For the function on the left of Figure~\ref{fig:efnsGm+2} this is easy:  $f$ vanishes on the copy of $G_{m+1}$ so this makes no contribution to $Lf(u)$, and the antisymmetry ensures the Neumann derivatives from the two copies of $G_m$ cancel, giving $Lf(u)=0$ which matches $\lambda f(u)$ in this case.

The function on the right of Figure~\ref{fig:efnsGm+2} requires slightly more explanation. We have $f(u)=-1$ at  both boundary points of the copy of $G_{m+1}$ glued at $u$. This uniquely defines the restriction of $f$ to this copy to be the negative of the sum of the basis elements from Lemma~\ref{lem:efnsGm+1}.  In particular it is $-2\kappa h$ on the copy of $G_m$ inside this $G_{m+1}$ and its Neumann derivative at $u$ is $-2\partial h_{m-1}(1)$.  On the copies of $G_m$ we have that $f$ is the antisymmetric function seen on the right in Figure~\ref{fig:efnsGm} minus $2\kappa h$, where $h$ is the eigenfunction from the left of the same figure.  The Neuman derivative of the antisymmetric function is $-\partial h_{m-2}(1)$ and the Neumann derivative of the eigenfunction is $\partial h_{m-2}(0)$, both of which were determined in Proposition~\ref{prop:efnsGm}, giving a total of $-2\partial h_{m-2}(1)+4\kappa \partial h_{m-2}(0)$ from the two copies of $G_m$. However, $2\kappa\partial h_{m-2}(0)=-\partial h_{m-1}(0)$ by Lemma~\ref{lem:efnsGm+1}. Thus the sum of the Neumann derivatives from $G_{m-1}$ and the two copies of $G_m$ is $-2(\partial h_{m-1}(1)+\partial h_{m-2}(1)+\partial h_{m-1}(0))$ and this is $-\lambda=\lambda f(u)$ by the formula established in Corollary~\ref{cor:sumofNeumderivs}.

To show all solutions of  $Lf=\lambda f$ on $G_{m+2}\setminus \partial G_{m+2}$ are in the span of those described above, notice that the restriction of $f$ to the copies of $G_m$ must be linear combinations of the two functions in Figure~\ref{fig:efnsGm} by Proposition~\ref{prop:efnsGm}. Continuity at $u$ then restricts their boundary values and value at $u$ to be a multiple of those for the second function we have considered (on the right of Figure~\ref{fig:efnsGm+2}), so by subtracting this multiple we may assume $f$ is zero on $\partial G_{m+2}$ and at $u$.  It follows that the restriction of $f$ to each copy of $G_m$ is a multiple of the eigenfunction $h$.  Moreover, $f(u)=0$ is the value at both boundary points of the copy of $G_{m+1}$ in $G_{m+2}$. Since $\lambda$ is not a Dirichlet eigenvalue of this subgraph by Lemma~\ref{lem:efnsGm+1} we have $f\equiv0$ the $G_{m+1}$ copy.  This shows $Lf(u)=\lambda f(u)=0$ is the sum of the Neumann derivatives of the multiples of $h$ on the copies of $G_m$, and since $h$ has non-zero Neumann derivative the only possibility is that the multiples are equal in magnitude and opposite in sign, whence $f$ is a multiple of the function on the left of Figure~\ref{fig:efnsGm+2}.
\end{proof}

\begin{figure}
\centering
\begin{tikzpicture}
[vertex/.style={circle,fill,inner sep=2pt},node distance=1.5cm]
% First graph
\node[vertex] (bdychain11) at (0,0) [label=above:$0$] {};
\node[vertex] (bdychain12) [right=of bdychain11, label=above:$1$]{};
\node[vertex] (bdychain13) [right=of bdychain12] [label=above:$0$] {};
\node[vertex] (bdychain14) [right=of bdychain13] [label=above:$-1$] {};
\node[vertex] (bdychain15) [right=of bdychain14] [label=above:$0$] {};
\node[vertex] (decoration1) [below=of bdychain13] [label=right:$0$] {};

\draw (bdychain11)-- node[below]{$$} (bdychain12) -- node[below]{$$} (bdychain13) -- node[below]{$$} (bdychain14) -- node[below]{$$} (bdychain15);
\draw (bdychain13) edge[bend right] node[left]{$$} (decoration1) edge[ bend left] node[right]{$$} (decoration1);

\path[-, every loop/.style= {looseness=10, distance=20, in=300,out=240}]   (bdychain12) edge [loop below] node{$$}();
\path[-, every loop/.style= {looseness=10, distance=20, in=300,out=240}]   (bdychain14) edge [loop below] node{$$}();
\path[-, every loop/.style= {looseness=10, distance=20, in=300,out=240}]   (decoration1) edge [loop below] node{$$}();

\node (leftpiece) [below=of bdychain12,yshift=10mm] [label=below:$h$]{};
\node (leftpiece) [below=of bdychain14,yshift=10mm] [label=below:$-h$]{};

% Second graph
\node[vertex] (bdychain21) [right=of bdychain15] [label=above:$1$] {};
\node[vertex] (bdychain22) [right=of bdychain21, label=above:$2\kappa$]{};
\node[vertex] (bdychain23) [right=of bdychain22] [label=above:$-1$] {};
\node[vertex] (bdychain24) [right=of bdychain23] [label=above:$2\kappa$] {};
\node[vertex] (bdychain25) [right=of bdychain24] [label=above:$1$] {};
\node[vertex] (decoration2) [below=of bdychain23] [label=right:$0$] {};

\draw (bdychain21)-- node[below]{$$} (bdychain22) -- node[below]{$$} (bdychain23) -- node[below]{$$} (bdychain24) -- node[below]{$$} (bdychain25);

\draw (bdychain23) edge[bend right] node[left]{$$} (decoration2) edge[ bend left] node[right]{$$} (decoration2);

\path[-, every loop/.style= {looseness=10, distance=20, in=300,out=240}]   (bdychain22) edge [loop below] node{$$}();
\path[-, every loop/.style= {looseness=10, distance=20, in=300,out=240}]   (bdychain24) edge [loop below] node{$$}();
\path[-, every loop/.style= {looseness=10, distance=20, in=300,out=240}]   (decoration2) edge [loop below] node{$-2\kappa h$}();
\end{tikzpicture}
\caption{Solutions of $Lf=\lambda f$ on $G_{m+2}\setminus \partial G_{m+2}$.}
\label{fig:efnsGm+2}
\end{figure}

\begin{lemma}\label{lem:inductionforefns}
For $n\geq m+3$ the only solutions to $Lf=\lambda f$ on $G_n\setminus\partial G_n$ that are not DN eigenfunctions are as follows:
\begin{enumerate}[(1)]
\item If $n-m$ is odd, the functions described in Corollary~\ref{cor:allefnsonGmodd}.
\item If $n-m$ is even, a single eigenfunction obtained by copying the eigenfunction that is Dirichlet but not Neumann on $G_{n-2}$ onto both copies of $G_{n-2}$ in a $\Phi_n$-antisymmetric fashion and setting $f\equiv0$ on the copy of $G_{n-1}$. This has Neumann derivatives $\pm\partial h_{m-2}(0)$ at its boundary points.
\end{enumerate}
\end{lemma}
\begin{proof}
We induct on $n$ and use the fact that the restriction of $f$ to the copies of $G_{n-2}$ and $G_{n-1}$ satisfy the same equation so have the form described in Lemma~\ref{lem:efnsGm+1} and Lemma~\ref{lem:efnsGm+2} or, by the inductive hypothesis, the form given in the statement of this lemma.

The easier situation is when $n-m$ is odd. By subtracting the known eigenfunctions from Corollary~\ref{cor:allefnsonGmodd} we can assume $f=0$ on $\partial G_n$.  We know the restriction to the copies of $G_{n-2}$ is one of the functions from Lemma~\ref{lem:efnsGm+1} in the base case $n=m+3$ or, by the inductive hypothesis, one of the functions from Corollary~\ref{cor:allefnsonGmodd}  if $n\geq m+5$. In either case we see that $f=0$ at the boundary point implies $f\equiv0$ on the copy of $G_{m-1}$ that includes this boundary point, so the function is DN.

The argument when $n-m$ is even is a little more complicated. We first consider $n=m+4$ in which the restriction of $f$ to the copies of $G_{n-2}=G_{m+2}$ must be as in Figure~\ref{fig:efnsGm+2}. This implies $f$ has the same value on $\partial G_n$ and at the gluing point $u$.

We show by contradiction that we cannot have $f(u)\neq0$, for which by scaling it suffices to consider the case $f(u)=1$.  If we did, then the restriction to the copies of $G_{m+2}$ is the function on the right of  Figure~\ref{fig:efnsGm+2}, which has Neumann derivative $\partial h_{m-2}(1)+2\kappa \partial h_{m-2}(0)$ at both boundary points.  From the formula in the proof of Lemma~\ref{lem:efnsGm+1} this is $\partial h_{m-2}(1)-\partial h_{m-1}(0)$, so the contribution to the Laplacian at $u$ of the two copies is $2(\partial h_{m-2}(1)-\partial h_{m-1}(0))$.  At the same time, $f(u)=1$ implies the the restriction of $f$ to the copy of $G_{n+3}$ has value $1$ at both boundary points.  By the inductive hypothesis this is the sum of the functions in Corollary~\ref{cor:allefnsonGmodd}, so has Neumann derivative $2\partial h_{m-1}(1)$ at the gluing point.  Thus $Lf(u)=2\bigl(\partial h_{m-2}(1)-\partial h_{m-1}(0)+\partial h_{m-1}(1)\bigr)$.  According to Corollary~\ref{cor:sumofNeumderivs} this is $\lambda-4\partial h_{m-1}(0)$ and since $\partial h_{m-1}(0)$ is non-zero from Lemma~\ref{lem:adjpointszeroimpliesvanishes} we see that $Lf(u)\neq \lambda=\lambda f(u)$.  Thus there is no solution built from these solutions on $G_{n-2}$.

It follows that we must have $f=0$ at $u$ and on $\partial G_{m+4}$.   We see that the same is true in the case $n\geq 6$ by the inductive hypothesis, because then the restriction of $f$ to both copies of $G_{n-2}$ must be a multiple of the Dirichlet eigenfunction. In this situation the restriction of $f$ to the copy of $G_{n-1}$ has both boundary values equal to zero, so by the inductive assumption this is a DN eigenfunction and its Neumann derivative makes no contribution to the Laplacian at $u$.  Thus the equation $Lf(u)=\lambda f(u)=0$ says the Neumann derivatives from $f$ restricted to the copies of $G_{n-2}$ must cancel, and since we know they are multiples of $\partial h_{m-2}(0)$ (by Lemma~\ref{lem:efnsGm+2} in the case $n=m+4$ and the inductive hypothesis if $n\geq m+6$) and this is non-zero by Lemma~\ref{lem:adjpointszeroimpliesvanishes}, we conclude that the multiples are equal magnitude and opposite in sign, closing the induction.
\end{proof}

\begin{theorem}\label{thm:structure of DN efns}
The Dirichlet eigenfunction on $G_n$ with eigenvalue $\lambda$ a root of $\gamma_m$ for some $m<n$ have the following structure:
\begin{enumerate}[(1)]
\item\label{item:allefnsDNifodd} If $n-m$ is odd then all Dirichlet eigenfunctions are also Neumann eigenfunctions. There are no eigenfunctions for $n=m+1$.
\item\label{item:antisymefns} If $n-m$ is even then there is a one-dimensional space of eigenfunctions that are Dirichlet but not Neumann. The eigenfunctions are $\Phi_n$-antisymmetric.  A basis element is given by decomposing the shortest path between the boundary points of $G_n$ into copies of $G_m$, placing copies of the Dirichlet eigenfunction $h$ with alternating signs along these copies of $G_m$ and setting $f\equiv0$ on decorations attached to the boundary points of the copies.  When $n=m+2$ this is the only eigenfunction.
\item\label{item:constructionofDN} Dirichlet-Neumann eigenfunctions $f$ on $G_n$ can be constructed in the following manner, and all DN eigenfunctions arise from this construction.
\begin{enumerate}[(i)]
\item Taking $f$ to coincide with DN eigenfunctions on each copy of $G_{n-2}$ and on the copy of $G_{n-1}$. 
\item If $n-m$ is odd and $n\geq m+3$, setting $f\equiv0$ on the copy of $G_{n-2}$ and taking the restriction of $f$ to the copy of $G_{n-1}$ to be an eigenfunction from the space in~\ref{item:antisymefns} above.
\end{enumerate}
\end{enumerate}
\end{theorem}

\begin{proof}
Statements~\eqref{item:allefnsDNifodd} and~\eqref{item:antisymefns} have already been established in the preceding results of this section, with the exception of the statement that the basis element in~\ref{item:antisymefns} is a sequence of copies of $h$  with alternating signs that vanishes on other decorations.  However, this latter is already seen in Lemma~\ref{lem:efnsGm+2} as shown on the left of Figure~\ref{fig:efnsGm+2} and follows inductively for larger $n$ using the fact that the Dirichlet but not Neumann eigenfunction on $G_n$ constructed in Lemma~\ref{lem:inductionforefns} consists of antisymmetrically arranged copies of the corresponding eigenfunction on $G_{n-2}$.

For statement~\eqref{lem:inductionforefns}, the fact that the constructions give DN eigenfunctions is elementary.  We need only check the equation $Lf(u)=\lambda f(u)$ at the gluing point $u$, and in both cases $f(u)=0$. In the first construction also all Neumann derivatives are zero, so $Lf(u)=0$. In the second construction we have $Lf(u)=0$ because the (non-zero) Neumann derivatives cancel due to the antisymmetry of the eigenfunction on the copy of $G_{n-1}$.

It is a little more challenging to check that these are the only DN eigenfunctions. Observe that we can assume $n\geq m+3$ because we found no DN eigenfunctions in the solutions of $Lf=\lambda f$ on $G_n$ for $n=m+1$ or $n=m+2$.  So we are in the situation described in Lemma~\ref{lem:inductionforefns}.  The restriction of a DN eigenfunction $f$ on $G_n$ to the copies of $G_{n-2}$ and the copy of $G_{n-1}$ inside $G_n$ satisfies $Lf=\lambda f$ on these copies so is as described in the previous results.

In the case that $n-m$ is even the only functions  in Lemma~\ref{lem:inductionforefns} that have DN conditions at one boundary point are DN eigenfunctions; $f$ must be one of those on each copy of $G_{n-2}$ or must vanish on $G_{n-2}$, and in either case its restriction to the copy of $G_{n-1}$ is also DN, so the function arises from the construction~\eqref{item:constructionofDN}(i).

If $n-m$ is odd we instead have that the restriction of $f$ to the copies of $G_{n-2}$ is one of the functions from Corollary~\ref{cor:allefnsonGmodd}. A priori, it could be that these are arranged so as to have DN boundary conditions and value $f(u)=1$, but in this case we would need the restriction of $f$ to $G_{n-1}$ to have value $1$ at both boundary points. Lemma~\ref{lem:inductionforefns} precludes this possibility for $n\geq m+5$, as then $n-1\geq m+4$ and $n-m$ is odd, so the only solutions of $Lf=\lambda f$ on $G_{n-1}$ have Dirichlet boundary conditions and  cannot match the condition $f(u)=1$.  In the remaining case $n=m+3$ there is a solution on $G_{n-1}=G_{n+2}$ with value $1$ at both boundary points: it is the function on the right in Figure~\ref{fig:efnsGm+2}.   However, at both boundary points this has Neumann derivative $\partial h_{m-2}(1)+2\kappa\partial h_{m-2}(0)=\partial h_{m-2}(1)-\partial h_{m-1}(0)$, where we used the formula from the proof of Lemma~\ref{lem:efnsGm+1}.  The Neumann derivative from each boundary point of the restriction of $f$ to a copy of $G_{n-2}$ is $\partial h_{m-1}(1)$. Summing these we have $Lf(u)=2\bigl(\partial h_{m-2}(1)-\partial h_{m-1}(0)+\partial h_{m-1}(1)\bigr)=\lambda-4\partial h_{m-1}(0)$ from Corollary~\ref{cor:sumofNeumderivs}, and therefore $Lf(u)\neq\lambda=\lambda f(u)$ because $\partial h_{m-1}(0)\neq0$ by Lemma~\ref{lem:adjpointszeroimpliesvanishes}.

Having established that for a DN eigenfunction we cannot have the restriction of $f$ to the copies of $G_{n-2}$ to be non-zero multiples of the functions in Corollary~\ref{cor:allefnsonGmodd} we conclude from $f\equiv0$ on these sets that they make no contribution to the Laplacian $Lf(u)$. It follows from this and $Lf(u)=\lambda f(u)=0$ that the Neumann derivatives of the restriction of $f$ to the copy of $G_{n-1}$ must cancel when its boundary points are identified.  In any case this function must be a Dirichlet eigenfunction on $G_{n-1}$. If it is DN then $f$ arises from the construction~\eqref{item:constructionofDN}(i).  If it is Dirichlet but not DN then by Lemma~\ref{lem:efnsGm+2} for the case $n=m+3$ (so $n-1=m+2$) or Lemma~\ref{lem:inductionforefns} the cancellation of the Neumann derivatives ensures it arises by the construction~\eqref{item:constructionofDN}(ii).
\end{proof}

\begin{corollary}
Dirichlet Neuman eigenfunctions on $G_n$ are periodic on loops. Those with eigenvalues that are roots of $\gamma_m$ have period two copies of $G_m$ and are supported on loops of copies of $G_m$.
\end{corollary}

\begin{theorem} \label{thm:factorizationofcn}
The powers in the factorization of $c_n$ may be given explicitly as
\begin{gather}
	c_n=\gamma_n\prod_{k=1}^{n-1} \gamma_k^{S_{n-k}}, \quad\text{where}  \label{eqn:factorcn}\\
	S_{n} = \dfrac{9 + 23(-1)^{n} + 2^{2+n} - 6n(-1)^{n}}{36}.
	\end{gather}
The roots of $\gamma_k$ are simple, so the multiplicity of an eigenvalue is determined precisely by $S_{n-k}$ where $G_k$ is the smallest of the graphs for which the eigenvalue occurred.
%Moreover the multiplicities of the Dirichlet-Neumann eigenvalues.???
\end{theorem}
\begin{proof}
From Theorem~\ref{thm:structure of DN efns} we know exactly how a root $\lambda$ of $\gamma_k$ occurs as a Dirichlet eigenvalue on $G_n$, and hence as a root of $c_n$. In particular, we can use this to obtain a recursion for the powers $s_{n,k}$ in~\eqref{eqn:recursionforcn} as follows.

Fix $\lambda$ a root of $\gamma_k$. According to  Theorem~\ref{thm:structure of DN efns} if $n-k$ is odd all eigenfunctions with eigenvalue $\lambda$ are Dirichlet-Neumann, so $s_{n,k}$ is the dimension of the DN eigenspace.  If $n-k$ is even there is one eigenfunction that is Dirichlet but not Neumann; considering the Neumann derivative clearly shows it is linearly independent of the DN eigenfunctions, so $s_{n,k}-1$ is the dimension of the DN eigenspace.

We also know exactly how DN eigenfunctions arise on $G_n$. If $n-k$ is even this is only by copying DN eigenfunctions from $G_{n-2}$ to the two copies of this graph in $G_n$ or copying DN eigenfunctions on $G_{n-1}$ to the single copy of this graph in $G_n$. The fact that these copies are disjoint aside from intersecting at the gluing point, where all functions concerned are zero, ensures the DN eigenfunctions thus constructed are linearly independent.  Hence for $n-k$ even the dimension of DN eigenfunctions on $G_n$ is twice that of DN eigenfunctions on $G_{n-2}$ plus that of DN eigenfunctions on $G_{n-1}$.  Writing this in terms of the indices $s_{m,n}$ gives $s_{n,k}-1 = s_{n-1,k}+2(s_{n-2,k}-1)$, or $s_{n,k}=s_{n-1,k}+2s_{n-2,k}+1$.

If $n-k$ is odd the same construction applies for DN eigenfunctions, but there is one additional eigenfunction from the construction~\eqref{item:constructionofDN}(ii) of Theorem~\ref{thm:structure of DN efns}.  It is linearly independent from those previously constructed because it was not DN on the copy of $G_{n-1}$ before the boundary points were identified.  Writing the dimension of the DN eigenspace as before gives $s_{n,k}=1+(s_{n-1,k}-1)+2s_{n-2,k}=s_{n-1,k}+2s_{n-2,k}$.

We also know that $s_{k+1,k}=0$ and $s_{k+2,k}=1$.  We rewrite the preceding as the following recursion
\begin{equation*}
s_{n,k} =
	\begin{cases}
	 s_{n-1,k} + 2s_{n-2,k} - 1  & \text{ if $k \leq n-3$ and $n-k$ is even,}\\
	 s_{n-1,k} + 2s_{n-2,k}  & \text{ if $k \leq n-3$ and $n-k$ is odd,}\\
	1 &\text{ if $k=n-2$,}\\
	 0 & \text{ if $k=n-1$,}
	\end{cases}
\end{equation*}
and view it as a recursion in $n$ beginning at $n = k+3$, with initial data $s_{k+1,k} = 0$ and $s_{k+2,k} =1$. Then $s_{n,k} = s_{n-k +1, 1}$ because they satisfy the same recursion with the same initial data. Hence, $s_{n,k} = S_{n-k}$ where for $n\geq2$, $S_n$ satisfies the recursion
\begin{equation}\label{eqn:recusionforS}
S_{n+1} = S_n + 2S_{n-1} - \frac{1}{2}(1-(-1)^n)
\end{equation}
with $S_1 =0$ and $S_2 = 1$.
The formula for $S_n$ given in the statement of the proposition satisfies this recursion because
\begin{align*}
	\lefteqn{36(S_{n-1}+2S_{n-2}-\frac{1}{2}(1-(-1)^{n-1})}\quad&\\
	& = 9+2\cdot9+23(-1)^{n-1}+2\cdot23(-1)^{n-2}+2^{n+1}+2\cdot2^{n} \\
	&\quad- 6(n-1)(-1)^{n-1}-2\cdot6(n-2)(-1)^{n-2} -18(1+(-1)^n)\\
	& =9 + 23(-1)^{n} + 2^{n+2} - 6n(-1)^{n}.  \qedhere
	\end{align*}
\end{proof}

\subsection{Dynamics for the $\gamma_n$ factors}\label{subsec-dyn-gamma-n}

The recusions we have for the $c_n$ imply recursions for the factors $\gamma_n$.
\begin{proposition}\label{prop:gammarecurse}
The polynomials $\gamma_n$, $n\geq3$ may be computed recursively from the initial polynomials $\gamma_1=c_1=\lambda-2$, $\gamma_2=c_2=\lambda^2-6\lambda+4$ and the relation
\begin{align*} %\label{eqn:gammarecursion}
	\bigl( \gamma_n - 2 \eta_n \bigr)
	 \mkern-15mu \prod_{0\leq 2j\leq n-4 }  \mkern-15mu \gamma_{n-2j-3} %}\quad &\\
	&=
	\bigl( \gamma_{n-1} - 2 \eta_{n-1} \bigr)
	\bigl( \gamma_{n-1} + 2 \eta_{n-1}\bigr)  \mkern-15mu
	\prod_{0\leq 2j\leq n-5} \mkern-12mu \gamma_{n-2j-4},
	\end{align*}
in which
\begin{equation*}
	\eta_n=\gamma_{n-1}  \prod_{0\leq2j\leq n-4}  \gamma_{n-2j-3}^{2^j}.
	\end{equation*}
\end{proposition}
\begin{proof}
From~\eqref{eqn:startpolys3} we know $\gamma_3=\lambda^4-12\lambda^3+42\lambda^2-44\lambda+8$ and can check by hand that it satisfies the given relation.  For $n\geq4$ we use the recursion~\eqref{eqn:cdynamics1} for $c_n$ from Proposition~\ref{prop:cdynamics}, which we rewrite in the following two forms, with the latter obtained from the former using the definition~\eqref{eqn:defnofgn} of $g_n$:
\begin{align}
\frac{c_n}{c_{n-2}}  -  2  c_{n-1}g_{n-2} &=  \Bigl( \frac{c_{n-1}}{ c_{n-3}}\Bigr)^2    -4c_{n-2} g_{n-1}, \notag\\	
\frac{c_n}{c_{n-2}}  -   \frac{2g_n}{ g_{n-2}} &=  \Bigl( \frac{c_{n-1}}{ c_{n-3}}    - \frac{2g_{n-1}}{g_{n-3}} \Bigr)	 \Bigl( \frac{c_{n-1}}{ c_{n-3}}    + \frac{2g_{n-1}}{g_{n-3} }\Bigr). \label{eqn:cgrecursion}
\end{align}

It is then useful to compare the powers of $\gamma_k$ that occur in each of the component expressions.  For $c_n/c_{n-2}$ the power of $\gamma_n$ is $1$ and the power of $\gamma_k$ for $1\leq k\leq n-2$ is
\begin{equation}\label{eqn:Sndiffs}
	S_{n-k}-S_{n-k-2}
	= \frac13 \bigl( 2^{n-k-2} +(-1)^{n-k+1} \bigr)
	\end{equation}
where the explicit expression is from Theorem~\ref{thm:factorizationofcn}.

From the formula~\eqref{eqn:defnofgn} for $g_n$ we have
\begin{equation*}
	\frac{g_n}{g_{n-2}}
	= \begin{cases}
	 c_2^{2^{(n-3)/2}} \prod_{0\leq 2j< n-3} \Bigl( \frac{c_{n-1-2j}}{c_{n-3-2j}} \Bigr)^{2^j} &\text{ if $n$ is odd,}\\
	 c_1^{2^{(n-2)/2}} \prod_{0\leq 2j< n-3} \Bigl( \frac{c_{n-1-2j}}{c_{n-3-2j}} \Bigr)^{2^j} &\text{ if $n$ is even.}
	\end{cases}
	\end{equation*}
The difference between odd and even $n$ only affects the powers of $c_1=\gamma_1$ and $c_2=\gamma_2$, requiring that we add $2^{(n-3)/2}$ to the formula for $k=2$ if $n$ is odd and $2^{(n-2)/2}$ to the formula for $k=1$ if $n$ is even.  Conveniently, these both modify the case when $n-k$ is odd, which is also different to that for even values of $n-k$ in the cases $k\geq3$ because in the former case the occurence of $(c_k/c_{k-2})^{2^{(n-k-1)/2}}$ in the product introduces an additional factor of $\gamma_k^{2^{(n-k-1)/2}}$ that is not present when $n-k$ is even.   Note that the amount added in the $k=1,2$ cases is consistent with this formula. Accordingly, the power of $\gamma_k$ in $g_n/g_{n-2}$ for $1\leq k\leq n-3$ is
\begin{gather*}
	2^{(n-k-1)/2}+ \sum_{0\leq 2j\leq n-k-3} 2^j \bigl(S_{n-k-1-2j}- S_{n-k-3-2j}\bigr) \text{ if $n-k$ is odd,}\\
	\sum_{0\leq 2j\leq n-k-3} 2^j \bigl(S_{n-k-1-2j}- S_{n-k-3-2j}\bigr) \text{ if $n-k$ is even.}
\end{gather*}
We also note that the power of $\gamma_{n-1}$ is $1$ and no other $\gamma_j$ with $j>n-3$ occurs.
Simplifying the series using~\eqref{eqn:Sndiffs} gives
\begin{align*}
	\lefteqn{\sum_{0\leq 2j\leq n-k-3} 2^j \bigl(S_{n-k-1-2j}- S_{n-k-3-2j}\bigr)} \quad&\\
	&= \frac13 \sum_{0\leq 2j\leq n-k-3} 2^j \bigl(  2^{(n-k-3-2j)} + (-1)^{n-k-2j} \bigr)\\
	&=\begin{cases}
	\frac13 \sum_0^{(n-k-3)/2} \bigl(  2^{(n-k-3-j)} - 2^j \bigr) &\text{ if $n-k$ is odd}\\
	\frac13 \sum_0^{(n-k-4)/2} \bigl(  2^{(n-k-3-j)} + 2^j \bigr) &\text{ if $n-k$ is even}
	\end{cases}\\
	&=\begin{cases}
	\frac13 \bigl(  2^{(n-k-2)}- 2^{(n-k-3)/2} -  (2^{(n-k-1)/2}-1) \bigr) &\text{ if $n-k$ is odd}\\
	\frac13  \bigl(  2^{(n-k-2)} - 2^{(n-k-2)/2} + ( 2^{(n-k-2)/2}-1)\bigr) &\text{ if $n-k$ is even}
	\end{cases}\\
	&=\begin{cases}
	 	\frac13 \bigl(  2^{(n-k-2)} + 1 \bigr) - 2^{(n-k-3)/2} &\text{ if $n-k$ is odd}\\
	      \frac13  \bigl(  2^{(n-k-2)} - 1\bigr) &\text{ if $n-k$ is even}
	      \end{cases}
\end{align*}
and adding back in the $2^{(n-k-1)/2}$ in the odd case finally leads to the following expression for powers of $\gamma_k$ in $g_n/g_{n-2}$ if $1\leq k\leq n-3$:
\begin{gather*}
	 \frac13 \bigl(  2^{(n-k-2)} + 1 \bigr) + 2^{(n-k-3)/2} \text{ if $n-k$ is odd,} \\
	  \frac13  \bigl(  2^{(n-k-2)} - 1\bigr) \text{ if $n-k$ is even.} %\label{eqn:gnpowers}
	\end{gather*}
Comparing this to~\eqref{eqn:Sndiffs} for powers of $\gamma_k$ for $c_n/c_{n-2}$ we obtain an expression for the left side of the recursion in~\eqref{eqn:cgrecursion}.
\begin{equation} \label{eqn:firsttermingammarecurse}
	\frac{c_n}{c_{n-2}} - \frac{2g_n}{g_{n-2}}
	= \Bigl( \gamma_n - 2 \gamma_{n-1} \prod_{0\leq2j\leq n-4}  \gamma_{n-2j-3}^{2^j} \Bigr)
		\prod_{j=1}^{n-3} \gamma_{n-j-2}^{(2^j - (-1)^j)/3}
	\end{equation}

The right side of the recursion  in~\eqref{eqn:cgrecursion} is the product of two terms like that on the left.  Reasoning as for that term we find them to be
\begin{align*}
	\Bigl( \frac{c_{n-1}}{ c_{n-3}}    - \frac{2g_{n-1}}{g_{n-3}} \Bigr)
	&=\Bigl( \gamma_{n-1} - 2 \gamma_{n-2} \prod_{0\leq2j\leq n-5}  \gamma_{n-2j-4}^{2^j} \Bigr)
	\prod_{j=1}^{n-4} \gamma_{n-j-3}^{(2^j - (-1)^j)/3} \\
	 \Bigl( \frac{c_{n-1}}{ c_{n-3}}    + \frac{2g_{n-1}}{g_{n-3} }\Bigr)
	&=\Bigl( \gamma_{n-1} + 2 \gamma_{n-2} \prod_{0\leq2j\leq n-5}  \gamma_{n-2j-4}^{2^j} \Bigr)
	\prod_{j=1}^{n-4} \gamma_{n-j-3}^{(2^j - (-1)^j)/3} 
	\end{align*}
The product has a factor
\begin{equation*}
	\prod_{j=1}^{n-4} \gamma_{n-j-3}^{2(2^j - (-1)^j)/3} 
	= \prod_{j=1}^{n-4} \gamma_{n-j-3}^{(2^{j+1} -2 (-1)^j)/3} 
	= \prod_{j=2}^{n-3} \gamma_{n-j-2}^{(2^j +2 (-1)^j)/3}
	\end{equation*}
so when we substitute these and~\eqref{eqn:firsttermingammarecurse} into~\eqref{eqn:cgrecursion} we may cancel most terms, leaving $\prod_{j=1}^{n-3} \gamma_{n-j-2}^{(-1)^j}$ on the right side.  To obtain our desired conclusion simply move the terms in this product with odd $j$ onto the left and kept those with even $j$ on the right.
\end{proof}

\begin{corollary}\label{cor:gammarecurse}
For $n\geq4$,
\begin{equation*}
	(\gamma_n-2\eta_n)\gamma_{n-3} 
	=(\gamma_{n-1}+2\eta_{n-1})(\gamma_{n-2}+2\eta_{n-2})(\gamma_{n-2}-2\eta_{n-2}).
	\end{equation*}
	\end{corollary}
\begin{proof}
Apply the relation in Proposition~\ref{prop:gammarecurse} twice.
\end{proof}
Implementing this recursion in Mathematica and applying a numerical root-finder we can get a sense of how the roots of the $\gamma_n$ are distributed depending on $n$, see Figure~\ref{fig:rootsofgamman}. Some structural features of this distribution will be discussed in Section~\ref{section:gaps}.

\begin{corollary}\label{cor:rootdynamics}
For $n\geq4$ the rational function $\zeta_n=\gamma_n/\eta_n$ has roots precisely at the roots of $\gamma_n$ and satisfies the recursion
\begin{equation*}
\zeta_n-2 = \Bigl( 1+\frac2{\zeta_{n-1}}\Bigr) (\zeta_{n-2}^2-4),
\end{equation*}
where the equality is valid at the poles in the usual sense of rational functions, and the initial data is
\begin{align}\label{eq:zetainitdata}
	\zeta_2=\frac{\lambda^2-6\lambda+4}{\lambda-2}, & & \zeta_3=\frac{\lambda^4-12\lambda^3+42\lambda^2-44\lambda+8}{\lambda^2-6\lambda+4}.
	\end{align}
\end{corollary}
\begin{proof}
Since $\eta_n$ is a product of powers of $\gamma_j$ where $j<n$ and these (by definition) have no roots in common with $\gamma_n$, the roots of $\zeta_n$ are precisely those of $\gamma_n$.  In order to see the recursion, observe from the definition (in Proposition~\ref{prop:gammarecurse}) that $\gamma_{n-3}\eta_n = \gamma_{n-1}\eta_{n-2}^2$, then write the recursion in Corollary~\ref{cor:gammarecurse} as
\begin{equation*}
(\zeta_n-2)\eta_n\gamma_{n-3} 
	=\Bigl(1+\frac2{\zeta_{n-1}}\Bigr)(\zeta_{n-2}+2)(\zeta_{n-2}-2)\gamma_{n-1}\eta_{n-2}^2.
\end{equation*}
This expression involves polynomials.  Cancellation of the the common factors leaves a recursion of rational functions of the desired type.
\end{proof}

\begin{figure}
\includegraphics{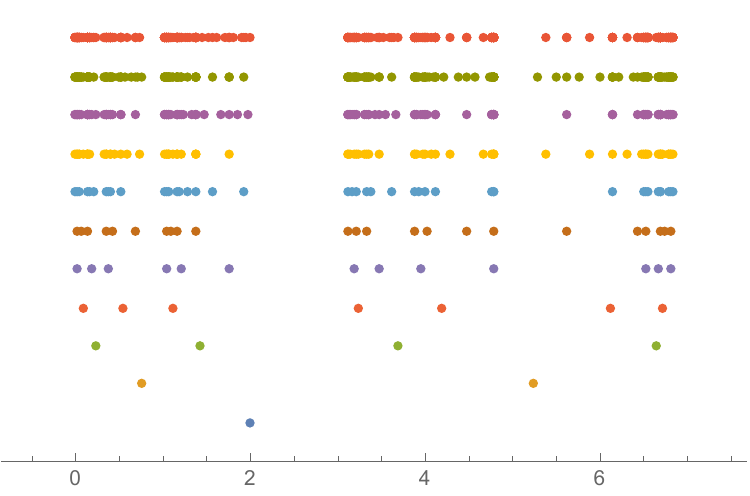}
\caption{Roots of $\gamma_n$ for $n=1,\dotsc,11$ ($n$ increases on the vertical axis).}\label{fig:rootsofgamman}
\end{figure}

\begin{proposition}\label{prop:degofgammanadeta}
The degree of $\gamma_n$ is
\begin{align*}
	\degree(\gamma_n)
	&=  \frac2{\sqrt{7}}\biggl( \rho_1^n \cos\bigl(\phi+\frac{2\pi}3\bigr) + \rho_2^n \cos\bigl(\phi + \frac{4\pi}3 \bigr)
	+\rho_3^n \cos\phi  \biggr)
	\end{align*}
where $\phi=\frac13\arctan(-3\sqrt{3})$ and
\begin{align*}
	 \rho_1=\frac13\Bigl(1 - 2 \sqrt{7}\cos \phi \Bigr),&& \rho_2=\frac13\Bigl(1 - 2 \sqrt 7 \cos\bigl(\phi + \frac{2 \pi}3\bigr) \Bigr), && \rho_3=\frac13\Bigl(1 + 2 \sqrt7 \cos\bigl(\phi + \frac\pi3\bigr) \Bigr).
	\end{align*}
Moreover the degrees of $\gamma_n$ and $\eta_n$ are related by
\begin{equation}\label{eqn:degetafromgamma}
	\degree(\eta_n) = \degree(\gamma_n) -2^{\lfloor\frac n2\rfloor}
	\end{equation}
where $\lfloor\frac n2\rfloor$ is the greatest integer less than $\frac n2$.
\end{proposition}
\begin{proof}
Observe that $\eta_1=\gamma_0$ has degree $1$ and $\eta_2=\gamma_1$ has degree $2$, while $\gamma_2$ has degree $4$. This shows  that~\eqref{eqn:degetafromgamma} holds
for $n=1,2$, and we suppose inductively that this holds for all $k\leq n-1$.  Examining the recursion in Corollary~\eqref{cor:gammarecurse} we see from the inductive hypotheses  that each bracketed term on the right has the same degree as its included $\gamma$ term, and therefore that
\begin{equation}\label{eqn:gammaandetadegreerecursion}
	\degree(\gamma_n-2\eta_n) = \degree(\gamma_{n-1}) + 2\degree(\gamma_{n-2}) - \degree(\gamma_{n-3}).
	\end{equation}
However $\gamma_{n-3}\eta_n = \gamma_{n-1}\eta_{n-2}^2$ and thus there is a similar recursion
\begin{align}
	\degree(\eta_n)
	&= \degree(\gamma_{n-1}) + 2\degree(\eta_{n-2}) -\degree(\gamma_{n-3}) \notag\\
	&= \degree(\gamma_{n-1}) + 2\degree(\gamma_{n-2})-2^{\lfloor( n-2)/2\rfloor+1} -\degree(\gamma_{n-3}),\label{eqn:etarecusion}
	\end{align}
where we have substituted the inductive hypothesis~\eqref{eqn:degetafromgamma} to obtain the second expression.
Comparing this to~\eqref{eqn:gammaandetadegreerecursion} proves that $\degree(\eta_n)<\degree(\gamma_n)$ and thereby reduces~\eqref{eqn:gammaandetadegreerecursion} to
\begin{equation}\label{eqn:gammadegreerecursion}
	\degree(\gamma_n) = \degree(\gamma_{n-1}) + 2\degree(\gamma_{n-2}) - \degree(\gamma_{n-3}).
	\end{equation}
Comparing this to~\eqref{eqn:etarecusion} proves that~\eqref{eqn:degetafromgamma} holds for $k=n$ and therefore for all $n$  by induction.

The recursion in~\eqref{eqn:gammadegreerecursion} can be solved by writing it as a matrix equation and computing an appropriate matrix power.  The matrix involved has characteristic polynomial $\rho^3-\rho^2-2\rho+1$, the roots $\rho_j$, $j=1,2,3$ of which are as given in the statement of the lemma.  The rest of the proof is standard.
\end{proof}

\section{KNS Spectral Measure}\label{sec-KMS}
For a sequence of graphs convergent in the metric~\ref{eq:pGHdist} the Kesten--von-Neumann--Serre (KNS) spectral measure, defined in~\cite{GZ04}, is the weak limit of the (Neumann) spectral measures for the graphs in the sequence.  In particular, for a blowup $G_\infty$ it is  the limit of the normalized sum of Dirac masses $\delta_{\lambda_j}$ at eigenvalues of the Laplacian $L_n$ on $G_n$, repeated according to their multiplicity.   Since the measure does not depend on which blowup $G_\infty$ we consider, we will henceforth just refer to the KNS spectral measure.  Note that by Theorem~\ref{thm:orbitalSareblowups} this is also the KNS spectral measure of the orbital Schreier graphs of the Basilica that do not have four ends.

Our first observation regarding the KNS spectral measure is that we can study it using the limit of the spectral measure for the Dirichlet Laplacian on $G_n$, or even the limit of the measure on Dirichlet-Neumann eigenfunctions on $G_n$.
\begin{lemma}\label{lem:computeKNSfromDirichletNeumann}
The KNS spectral measure is the weak limit of the spectral measure for the Dirichlet Laplacian on $G_n$, which is given by
\begin{equation}\label{eqn:spectralmsr}
	\chi_n = \frac1{V_n-2}   \sum_{\{\lambda_j:c_n(\lambda_j)=0\}} \delta_{\lambda_j} \\
	=  \sum_{k=1}^n \sum_{\{\lambda_j:\gamma_k(\lambda_j)=0\}} \frac{S_{n-k}}{V_n-2}\delta_{\lambda_j}.
	\end{equation}
Moreover, the support of the KNS spectral measure is contained in the closure of the union over $n$ of the set of Dirichlet-Neumann eigenvalues for the  Laplacian on $G_n$.
\end{lemma}
\begin{proof}
From Theorem~\ref{thm:structure of DN efns} the number of eigenfunctions of $L_n$ that are Dirichlet but not Neumann is no larger than $\deg(\gamma_n)+\frac n2$. Accordingly the number that are Neumann but not Dirichlet-Neumann does not exceed $2+\deg(\gamma_n)+\frac n2$.  But from Proposition~\ref{prop:degofgammanadeta} the degree of $\gamma_n$ is bounded by a multiple of $\rho^n$ for some $\rho<2$ (because we can check all $\rho_j<2$). The number of eigenvalues of $L_n$ grows like $2^n$ from Lemma~\ref{csgetdegrees}, so the proportion of eigenvalues corresponding to eigenfunctions that are not Dirichlet-Neumann is bounded by a multiple of $(\rho/2)^n$  and makes no contribution to the mass in the limit.  It follows that we get the same limit measure whether we take the limit of the spectrum of the Neumann Laplacian $L_n$, or the Dirichlet Laplacian on $G_n$, or even the normalized measure on the eigenvalue corresponding to Dirichlet-Neumann eigenfunctions.

The computation~\eqref{eqn:spectralmsr} can be justifed using the factorization in Theorem~\ref{thm:factorizationofcn} and the observation that the degree of $c_n$ is two less than the number of vertices of $G_n$, which was computed in Lemma~\ref{csgetdegrees}.  A graph of the spectral measure $\chi_{11}$ for $G_{11}$ is in Figure~\ref{fig:specmsr}.

For the final statement of the lemma, observe that if $\lambda$ is in the support of the KNS measure and $U$ is a neighborhood of $\lambda$ then $U$ has positive KNS measure and hence there is a lower bound on the $G_n$-spectral measure of $U$ for all sufficiently large $n$. We just saw that the proportion of the $G_n$ spectral measure that is not on Dirichlet-Neumann eigenvalues goes to zero as $n\to\infty$, so $U$ must contain a Dirichlet-Neumann eigenvalue. Thus the support of the KNS measure is in the closure of the union of the Dirichlet-Neumann spectra. 
\end{proof}

\begin{figure}
\centering
\includegraphics[width=12cm]{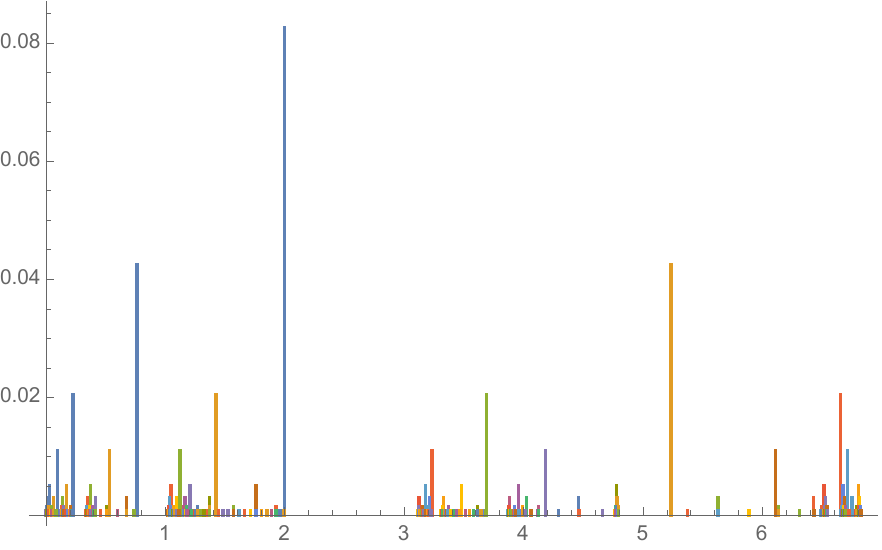}
\caption{Spectral measure $\chi_{11}$ of the Dirichlet Laplacian on $G_{11}$.}\label{fig:specmsr}
\end{figure}

We can compute the multiplicities and the degree of $c_n$, so it is easy to estimate the weights at the eigenvalues that occur as roots of $\gamma_k$.
\begin{lemma}\label{lem:SnasproportionofVn}
\begin{equation*}
	\Bigl| \frac{S_{n-k}}{V_n-2} -\frac16 2^{-k} \Bigr| \leq \frac{n+5}{2^{n+1}}.
	\end{equation*}
\end{lemma}
\begin{proof}
Compute using the formulas for $V_n$ and $S_n$ from Lemma~\ref{csgetdegrees} and Theorem~\ref{thm:factorizationofcn} that
\begin{align*}
	\Bigl|\frac{S_{n-k}}{V_n-2} -\frac16 2^{-k} \Bigr|
	&= \frac16\Bigl|    \frac{9 + (23-6(n-k))(-1)^{n-k} + 2^{n-k+2} }{2^{n+2}-3+(-1)^{n+1} }   - 2^{-k}   \Bigr|\\
	&= \frac16 \Bigl| \frac{9(1-2^{-k}) + (23-6(n-k)+2^{-k})(-1)^{n-k} }{2^{n+2}-3+(-1)^{n+1} } \Bigr|\\
	&\leq \frac{n-k+6}{2^{n+1}}\qedhere
	\end{align*}
\end{proof}

This tells us that for fixed $k$ and large $n\gg k$ the measure $\chi_n$ has atoms of approximately weight $2^{-k}/6$ at each eigenvalue of the Dirichlet Laplacian on $G_k$. 
\begin{corollary}\label{cor:KNSspectfromDirNeum}
The support of the KNS spectral measure is the closure of the union of the Dirichlet spectra of the $G_n$.
\end{corollary}
\begin{proof}
In Lemma~\ref{lem:computeKNSfromDirichletNeumann} we saw that the support of the KNS spectral measure is in the closure of the union of the Dirichlet-Neumann spectra, which is clearly contained in the closure of the union of the Dirichlet spectra.  

Conversely, if $\lambda$ is a Dirichlet eigenvalue on $G_n$ then there is a smallest $k\leq n$ so $\lambda$ is an eigenvalue of $G_k$. Sending $n\to\infty$ we find from Lemma~\ref{lem:SnasproportionofVn} that the KNS measure will have an atom of weight $\frac162^{-k}$ at $\lambda$, which is therefore in the support of the KNS measure.
\end{proof}

To get more precise statements comparing $\chi_m$ to the limiting KNS measure it is useful to fix $m$ and estimate the amount of mass in $\chi_n$ that lies on eigenvalues from $G_k$, $k>m$. Arguing as in the proof of Lemma~\ref{lem:computeKNSfromDirichletNeumann} we might anticipate that this proportion is, in the limit as $n\to\infty$, bounded by $(\rho/2)^m$, so that the eigenvalues from $G_m$ capture all but a geometrically small proportion of the limiting KNS spectral measure.  We want a more precise statement, for which purpose we establish the following lemma.

\begin{lemma}\label{lem:tailestforspectmsr}
If $\rho=\rho_j$ is one of the values in Proposition~\ref{prop:degofgammanadeta} then
\begin{align*}
	\sum_{k=m+1}^n S_{n-k} \rho^k
	&= \frac1{36} \rho^{m+1} \Bigl( 2^{n-m+2}\rho(\rho+1) + (5\rho^2-4\rho-18) (-1)^{n-m}  \\
	&\quad + 6\rho(2-\rho)  (-1)^{n-m}(n-m) + 9(2-\rho^2) \Bigr)
\end{align*}
\end{lemma}
\begin{proof}
Compute, using $S_0=1$, $S_1=0$ and the recursion~\eqref{eqn:recusionforS} for $S_n$, $n\geq2$, that
\begin{align*}
	\sum_{k=m+1}^{n+1} S_{n+1-k} \rho^k
	&= \rho^{n+1} + \sum_{k=m+1}^{n-1} S_{n+1-k} \rho^k\\
	&=\rho^{n+1} + \sum_{k=m+1}^{n-1}\Bigl(S_{n-k}+2S_{n-1-k} -\frac12\bigl( 1-(-1)^{n-k}\bigr)\Bigr) \rho^k\\
	&=\rho^{n+1}-\rho^n + \sum_{m+1}^n S_{n-k}\rho^k + \sum_{m+1}^{n-1} 2S_{n-1-k}\rho^k \\
	&	\quad-\frac{\rho^n-\rho^{m+1}}{2(\rho-1)} + (-1)^n\frac{(-\rho)^n -(-\rho)^{m+1}}{2(-\rho-1)}\\
	&=  \sum_{m+1}^n S_{n-k}\rho^k + \sum_{m+1}^{n-1} 2S_{n-1-k}\rho^k \\
	&\quad+ \rho^n \Bigl( \rho-1 -\frac\rho{\rho^2-1} \Bigr) +\frac{\rho^{m+1}}{2} \Bigl( \frac{1}{\rho-1}-\frac{(-1)^{n-m}}{\rho+1}\Bigr)
\end{align*}
and conveniently the coefficient of $\rho^n$ has a factor $(\rho^3-\rho^2-2\rho+1)$, and the values $\rho_j$ are precisely the roots of this equation (see the end of the proof of Proposition~\ref{prop:degofgammanadeta}). Thus we have a recursion for our desired quantity, with the form 
\begin{equation*}
	\sum_{k=m+1}^{n+1} S_{n+1-k} \rho^k
	=  \sum_{m+1}^n S_{n-k}\rho^k + \sum_{m+1}^{n-1} 2S_{n-1-k}\rho^k 
	+ \frac{\rho^{m+1}}{2} \Bigl( \frac{1}{\rho-1}-\frac{(-1)^{n-m}}{\rho+1}\Bigr).
	\end{equation*}
The homogeneous part of the solution is  $\bigl(c_1 2^{n-m}+c_2 (-1)^{n-m}\bigr)\rho^{m+1}$.  The inhomogeneous part has terms $c_3\rho^{m+1}$ and  $c_4(n-m)(-1)^{n-m}\rho^{m+1}$. It is easy to calculate that
\begin{align*}
	c_3=\frac{-1}{4(\rho-1)} = \frac{(2-\rho^2)}4\\
	c_4=\frac{1}{6(\rho+1)} = \frac{\rho(2-\rho)}{6}
	\end{align*}
where the latter expression in each formula is from $\rho^3-\rho^2-2\rho+1=0$.  Then one can compute $c_1$ and $c_2$ from the initial values $\sum_{m+1}^{m+1}S_{n-k}\rho^k=\rho^{m+1}$ and $\sum_{m+1}^{m+2}S_{n-k}\rho^k=\rho^{m+2}$, which themselves come from $S_0=1$, $S_1=0$, or directly verify that the expression in the lemma has these initial values.
\end{proof}

\begin{corollary}\label{cor:limitspectralmeasure}
In the limit $n\to\infty$ the proportion of the spectral mass of $G_n$ that lies on eigenvalues of $G_m$ is
\begin{equation*}
	\frac1{3\sqrt{7}}\sum_j \cos\bigl(\phi+\frac{2j\pi}3\bigr) \rho_j^2(\rho_j+1)\bigl(\frac{\rho_j}2\bigr)^m
	\end{equation*}
where $\phi=\frac13\arctan(-3\sqrt{3})$ as in Proposition~\ref{prop:degofgammanadeta}.
\end{corollary}
\begin{proof}
Dividing $\sum_{k=m+1}^n S_{n-k}\rho_j^k$  by $V_n-2=\frac16\bigl(2^{n+2}+(-1)^{n+1}-3\bigr)$, using the result of Lemma~\ref{lem:tailestforspectmsr} and sending $n\to\infty$ gives
\begin{equation*}
	\lim_{n\to\infty} \frac1{V_n-2}\sum_{k=m+1}^n S_{n-k}\rho_j^k
	=\frac16 \rho_j^2(\rho_j+1) \bigl(\frac{\rho_j}2\bigr)^m
	\end{equation*}
whereupon the result follows by substitution into the expression
\begin{equation*}
 	\sum_{k=m+1}^n S_{n-k} \degree(\gamma_k) = \frac{2}{\sqrt{7}}\sum_{j=1}^3 \cos\bigl( \phi+\frac{2j\pi}3 \bigr) \sum_{k=m+1}^n S_{n-k}\rho_j^k
\end{equation*}
from Proposition~\ref{prop:degofgammanadeta}.
\end{proof}

A slightly more involved computation gives a bound on the $m$ needed to obtain a given proportion of the KNS spectral measure.

\begin{theorem}\label{thm:howtogetmostofspect}
For any $\epsilon>0$ there is $m$ comparable to $|\log\epsilon|$ such that, for $n\geq m$, all but $\epsilon$ of the spectral mass of any $G_n$ is supported on eigenvalues of the Laplacian on $G_m$.
\end{theorem}
\begin{proof}
Decompose the sum~\eqref{eqn:spectralmsr} into the sum  $\sum_{k=1}^m$  over eigenvalues of the Laplacian on $G_m$  and $\sum_{m+1}^n$ of eigenvalues of the Laplacian on $G_n$ that are not in the $G_m$ spectrum. As in the previous proof, use Proposition~\ref{prop:degofgammanadeta} to write
\begin{equation*}
 	\sum_{k=m+1}^n S_{n-k} \degree(\gamma_k) = \frac{2}{\sqrt{7}}\sum_{j=1}^3 \cos\bigl( \phi+\frac{2j\pi}3 \bigr) \sum_{k=m+1}^n S_{n-k}\rho_j^k
\end{equation*}
and then estimate using Lemma~\ref{lem:tailestforspectmsr}.  From the specific values of $\rho_j$ in Proposition~\ref{prop:degofgammanadeta} one determines
\begin{align}
	\sum_{k=m+1}^n S_{n-k}\rho_1^k &\leq \frac1{36}|\rho_1|^{m+1} \Bigl( \frac13 2^{n-m+2} + 25(n-m)+10\Bigr),\notag\\
	\sum_{k=m+1}^n S_{n-k}\rho_2^k &\leq \frac1{36}|\rho_2|^{m+1} \Bigl( \frac23 2^{n-m+2} + 5(n-m)+36\Bigr),\label{eqn:estimatesforcontribstospectmsr}\\
	\sum_{k=m+1}^n S_{n-k}\rho_2^k &\leq \frac1{36}|\rho_3|^{m+1} \Bigl( \frac{11}2 2^{n-m+2} + 3(n-m) + 21 \Bigr).\notag
	\end{align}
The largest of the $|\rho_j|$ is $|\rho_3|$, so we bound the terms not containing $2^{n-m+2}$ by $(n-m+2)|\rho_3|^{m+1}$.  For the terms that do contain $2^{n-m+1}$  we use the readily computed fact that $|\rho_1|^{m+1}/3+2|\rho_2|^{m+1}/3\leq |\rho_3|^{m+1}/2$ for all $m$ and combine these to obtain
\begin{equation*}
	\sum_{k=m+1}^n S_{n-k} \degree(\gamma_k)
	\leq \frac2{\sqrt{7}} \rho_3^{m+1} \Bigl( \frac16 2^{n-m+2} +  (n-m+2)\Bigr). 
	\end{equation*}
The contribution to the KNS spectral measure is computed by dividing by $V_n-2=\frac16\bigl(2^{n+2}+(-1)^{n+1}-3\bigr)$, which was computed in Lemma~\ref{csgetdegrees}.  This is larger than $\frac16 2^{n+1}$ because $n\geq1$, so from the above reasoning
\begin{equation*}
	\sum_{k=m+1}^n \frac{S_{n-k}}{V_n-2} \degree(\gamma_k)
	\leq  \frac8{\sqrt{7}} \Bigl( 1+ 6(n-m+2)2^{-(n-m+2)} \Bigr)\Bigl(\frac{|\rho_3|}2\Bigr)^{m+1}
	\end{equation*}
but $l2^{-l}$ is decreasing with maximum value $\frac12$, so we readily obtain
\begin{equation*}
\sum_{k=m+1}^n \frac{S_{n-k}}{V_n-2} \degree(\gamma_k)
	\leq  \frac{12}{\sqrt{7}} \Bigl(\frac{|\rho_3|}2\Bigr)^{m+1}<\epsilon
\end{equation*}
provided $m\geq C|\log\epsilon|$, where $C$ is a constant involving $\log\rho_3$.  This estimate says that at most $\epsilon$ of the spectral mass can occur outside the spectrum of $G_m$ once $m$ is of size $C|\log\epsilon|$.
\end{proof}

\section{Cantor structure of the spectrum}\label{section:gaps}

Our recursions for $c_n$ and $\gamma_n$ provide a method for computing the spectra of the $G_n$ for small $n$.  Using a desktop computer we were able to compute them for $n\leq14$.  By direct computation from~\eqref{eqn:estimatesforcontribstospectmsr}, using $(n-m)2^{1-(n-m)}\leq1$, these eigenvalues constitute at least 39\% of the spectrum (counting multiplicity) of any $G_n$, and the asymptotic estimate from Corollary~\ref{cor:limitspectralmeasure} says that as $n\to\infty$ they capture approximately 76\% of the KNS spectral measure.   The result is shown in Figure~\ref{fig:gaps}.

\begin{figure}
\includegraphics[width=13cm]{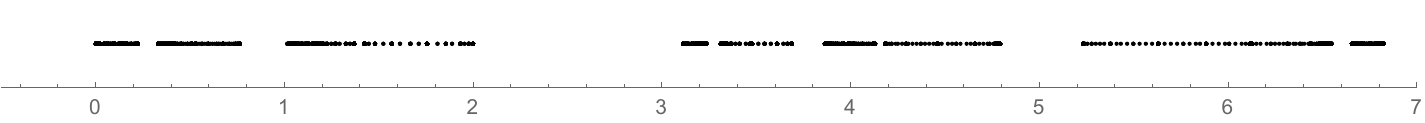}
\caption{The spectrum of $G_{14}$, illustrating gaps}\label{fig:gaps}
\end{figure}

Comparing Figures~\ref{fig:rootsofgamman}, \ref{fig:specmsr} and~\ref{fig:gaps} it appears that there are structural properties of the spectrum that are independent of $n$.  These should be features of the dynamics described in Section~\ref{section:dynamics}. The main result of this section is that the support of the KNS spectral measure is a Cantor set.  To prove this we use the dynamics  established in Corollary~\ref{cor:rootdynamics}, namely that for $n\geq4$ the eigenvalues first seen at level $n$, which are the roots of $\gamma_n=\gamma_n(\lambda)$, are also precisely the roots of $\zeta_n=\gamma_n/\eta_n$, which satisfies the recursion
\begin{equation}\label{eq:rootdynamics}
	\zeta_n-2=\Bigl( 1+\frac2{\zeta_{n-1}}\Bigr)(\zeta_{n-2}^2-4)
	\end{equation}
The initial data were given in~\eqref{eq:zetainitdata}.

We begin by describing an escape criterion under which future iterates  of~\eqref{eq:rootdynamics} do not get close to zero, and therefore cannot produce values in the spectrum. 

\begin{lemma}\label{lem:easyescape}
If  $n\geq 4$ and $|\zeta_{n-2}|>2$ and $|\zeta_{n-1}|>2$ then $|\zeta_m|\to\infty$ as $m\to\infty$.
\end{lemma}
\begin{proof}
Since $|\zeta_{n-1}|>2$  we have $1+\frac2{\zeta_{n-1}}>0$.  At the same time, $\zeta_{n-2}^2>4$, so $\zeta_n>2$ from~\eqref{eq:rootdynamics}. The same argument gives $\zeta_{n+1}>2$.  Now $\zeta_{n+1}>2$ implies $1+\frac2{\zeta_{n+1}}>1$ and thus from~\eqref{eq:rootdynamics}
\begin{equation*}
\zeta_{n+2}-2>\zeta_n^2-4=(\zeta_n-2)(\zeta_n+2)>4(\zeta_n-2).
\end{equation*}
This argument applies for all $\zeta_m$, $m\geq n+2$, so
\begin{equation*}
\zeta_m\geq 2^{m-n-2}\bigl(\min\{\zeta_n,\zeta_{n+1}\}-2\bigr)\to\infty
\end{equation*}
as $m\to\infty$.
\end{proof}

A similar analysis gives the following

\begin{lemma}\label{lem:pointsononeside}
Suppose $n\geq3$. For any $\delta\in(0,2)$ there is $k$ such that the region $|\zeta_{n-1}|>2$, $\zeta_{n}\in(2-\delta,2)$ contains a root of $\zeta_{n+2k}$.
\end{lemma}
\begin{proof}
Observe that if $|\zeta_{n+2j-1}>2|$ and $\zeta_{n+2j}\in(0,2)$ then from~\eqref{eq:rootdynamics} 
\begin{equation} \label{eq:zeta_n+2j+1>2}
	\zeta_{n+2j+1}-2=\Bigl(1+\frac2{\zeta_{n+2j}}\Bigr) (\zeta_{n+2j-1}-2)(\zeta_{n+2j-1}+2)
	\end{equation}
is a product of positive terms, so $\zeta_{n+2j+1}>2$.

Now suppose $\zeta_{n+2j+1}>2$.  Then the map $\zeta_{n+2j}\mapsto\zeta_{n+2j+2}$ is continuous and has $2\mapsto2$, so it 
takes an interval $(2-\delta_j,2)\subset(0,2)$ to an interval covering $(2-2\delta_j,2)$ because subsitution into~\eqref{eq:rootdynamics} gives
\begin{equation*}
	2-\delta_j\mapsto
	2- \delta_j (4-\delta_j) \Bigl(1+\frac2{\zeta_{n+2j+1}}\Bigr) 
	<2-2\delta_j.
	\end{equation*}

It follows from the above reasoning that if we begin with the region $|\zeta_{n-1}|>2$ and $\zeta_n\in(2-\delta,2)$ then the inductive statement that the $j^{\text{th}}$ iterated image satisfies $\zeta_{n+2j-1}>2$ and $\zeta_{n+2j}\in(0,2)$ for $1\leq j\leq k$ must fail before $k>\log_2\delta$.  Moreover it will fail because the image $\zeta_{n+2k}$ is an interval that strictly covers $(0,2)$, so there is a zero of $\zeta_{n+2k}$ in the required region.
\end{proof}

We now wish to proceed by analyzing a few steps of the orbit of a point $\tlambda$ at which $\zeta_n(\tlambda)=0$. This is complicated a little by the fact (immediate from~\eqref{eq:rootdynamics}) that $\zeta_{n+1}$ may have a pole at $\tlambda$.  We need a small lemma.

\begin{lemma}\label{lem:nopm2beforeroot}
If $\zeta_n(\tlambda)=0$ then $\zeta_m(\tlambda)\not\in\{-2,2\}$ for $m<n$.
\end{lemma}
\begin{proof}
Under the hypothesis there are no other $\gamma_m$ which vanish at $\tlambda$, so $\zeta_m$, $m<n$ has neither zeros nor poles at $\tlambda$; we use this fact several times without further remark.

There are some initial cases for which~\eqref{eq:rootdynamics} does not assist in computing $\zeta_m(\tlambda)$.   Evidently the statement of the lemma is vacuous if $n=1$.  If $n=2$ we compute $\tlambda=3\pm\sqrt{5}$, so  $\zeta_1(\tlambda)=\tlambda-2\not\in\{-2,2\}$.  If $n=3$ it is more useful to check that both $\zeta_1(\tlambda)=\pm2$ and $\zeta_2(\tlambda)=-2$ correspond to $\tlambda\in\{0,4\}$, while $\zeta_2(\tlambda)=-2$ implies $\tlambda=4\pm2\sqrt{2}$, because these are exactly the four solutions of  $\zeta_3(\tlambda)=2$. This verifies the lemma if $n=1,2,3$.  Moreover in the case $n\geq4$ the equivalence of $\zeta_2(\tlambda)\in\{-2,2\}$ with $\zeta_3(\tlambda)=2$ may also be used to exclude both of these possibilities, because if they hold then iteration of~\eqref{eq:rootdynamics} gives $\zeta_m(\tlambda)=2$ for all $m\geq3$ in contradiction to $\zeta_n(\tlambda)=0$.

Now with $n\geq4$ we use~\eqref{eq:rootdynamics} to see that if there were $3\leq m<n$ for which $\zeta_m(\tlambda)=-2$ then both $\zeta_{m+1}(\tlambda)=2$ and $\zeta_{m+1}(\tlambda)=2$, so that $\zeta_{m+k}(\tlambda)=2$ for all $k\geq1$ in contradiction to $\zeta_n(\tlambda)=0$.  Combining this with our initial cases, $\zeta_m(\tlambda)\neq-2$ for all $m<n$.

Finally, if there were an $m$ with $4\leq m<n$ and $\zeta_m(\tlambda)=2$ then taking the smallest such $m$ and applying~\eqref{eq:rootdynamics} would give $\zeta_{m-2}(\tlambda)=2$ because the other two roots are $\zeta_{m-1}(\tlambda)=-2$ and $\zeta_{m-2}(\tlambda)=-2$, both of which have been excluded. Since $m\geq4$ was minimal we have $m=4$ or $m=5$, but then either $\zeta_2(\tlambda)=2$ or $\zeta_3(\tlambda)=2$, both of which we excluded in our initial cases.
\end{proof}

\begin{theorem}\label{thm:valuesandgaps}
If $\zeta_n(\tlambda)=0$ then there is $\delta>0$ so that either the interval $I_-=(\tlambda-\delta,\tlambda)$ or the interval $I_+=(\tlambda,\tlambda+\delta)$ is a gap, meaning it does not intersect the Dirichlet Laplacian spectrum of $G_m$ for any $m\in\mathbb{N}$.  By contrast, there is a sequence $k_j\to\infty$ such that the other interval contains a sequence of Dirichlet eigenvalues  for the Laplacian on $G_{n+2k_j}$ that accumulate at $\tlambda$. 
\end{theorem}
\begin{proof}
Recall from Proposition~\ref{prop:efnsGm} zeros of $\gamma_n$ and thus of $\zeta_n$ are simple.  The definition of $\zeta_n=\gamma_n/\zeta_n$ ensures its zeros are also distinct from the zeros and poles of $\zeta_m$ $m<n$, so we may initially take $\delta$ so that $\zeta_n$ is positive on one of $I_-$, $I_+$ and negative on the other, and such that each $\zeta_m$, $m<n$ has constant sign on $I=(\tlambda-\delta,\tlambda+\delta)$.

Lemma~\ref{lem:nopm2beforeroot} ensures $\zeta_{n-2}(\tlambda)^2-4\neq0$, so~\eqref{eq:rootdynamics} and simplicity of the root of $\zeta_n$ at $\tlambda$ ensure $\zeta_{n+1}$ has a simple pole at $\tlambda$ if $n\geq3$.  For $n=1,2$ the same fact can be verified directly from the inital data~\eqref{eq:zetainitdata} for the dynamics.   In particular, $|\zeta_{n+1}(\lambda)|\to\infty$ as $\lambda\to\tlambda$. By reducing $\delta$, if necessary, we may assume $|\zeta_{n+1}(\lambda)|>2$ on $I\setminus\{\tlambda\}$.

We use the preceding to linearly approximate $\zeta_{n+j}$ for $j=2,3$.  Since~\eqref{eq:rootdynamics} is a dynamical system on rational functions we can linearize around a pole, but in order to use this dynamics we need $n\geq3$.  Temporarily write $t=\lambda-\tlambda$ and use $\simeq$ for equality up to $O(t^2)$ so simplicity of the root of $\zeta_n$ at $\tlambda$ implies there is a non-zero $\alpha$ with $\zeta_n(\lambda)\simeq \alpha t$ and the fact that $\zeta_{n-1}^2\neq4$ gives $\beta,\beta'$ with $\beta\neq0$ so $(\zeta_{n-1}^2-4)\simeq\beta+\beta't$.  Then we compute from~\eqref{eq:rootdynamics}:
\begin{align}
	\frac2{\zeta_{n+1}}
	&= \frac{2\zeta_n}{2\zeta_n+(\zeta_n+2)(\zeta_{n-1}^2-4)} \notag\\
	&\simeq \frac{2\alpha t}{2\alpha t + (\alpha t+2)(\beta+\beta't)} \simeq \frac\alpha\beta t, \label{eqn:zeta_n+1}
	\end{align}
and therefore
\begin{align}
	\zeta_{n+2}
	&=2+ \Bigl(1+\frac2{\zeta_{n+1}}\Bigr)(\zeta_n^2-4) \notag\\
	&\simeq 2+ \Bigl( 1+\frac\alpha\beta t\Bigr) (\alpha^2t^2-4)\simeq -2 -\frac{4\alpha}\beta t. \label{eqn:zeta_n+2}
	\end{align}
The preceding is valid for $n\geq3$, but if $n=2$ then $\tlambda\in\{3-\sqrt{5},3+\sqrt{5}\}$ and a  linearization of $2\zeta_3^{-1}$ like~\eqref{eqn:zeta_n+1} is readily computed from~\eqref{eq:zetainitdata} while the argument of~\eqref{eqn:zeta_n+2} is valid for $\zeta_4$. Moreover, if $n=1$ then $\tlambda=2$ and linearizations for both $2\zeta_2^{-1}$ and $\zeta_3$ can again be computed directly from~\eqref{eq:zetainitdata}. Thus~\eqref{eqn:zeta_n+1} and~\eqref{eqn:zeta_n+2} are valid for all $n\geq1$.
	
Since $\alpha$ and $\beta$ are non-zero, the linearizations show that $\zeta_{n+2}(\lambda)<-2$ for $t$ in an interval on the side of $0$ where $\frac\alpha\beta t>0$, meaning that $\lambda$ is on the corresponding side of $\tlambda$.  By reducing $\delta$, if necessary, we conclude $\zeta_{n+2}<-2$ on one of $I_+$ or $I_-$. At this point we have both $|\zeta_{n+1}(\lambda)|>2$ and $|\zeta_{n+2}(\lambda)|>2$ on exactly one of the two intervals $I_-$ or $I_+$, and since $n+1\geq2$ we can apply Lemma~\ref{lem:easyescape} to find that this interval does not contain zeros of $\zeta_m$ for any $m>n$.  Since it was also selected so as to not contain zeros of $\zeta_m$ for $m\leq n$ we have proved that one of these intervals is a gap.

Turning to the other interval, where $\frac\alpha\beta t<0$, we will need two more iterations of the linearized dynamics. The index $n$ is now large enough that we need only apply~\eqref{eq:rootdynamics} to~\eqref{eqn:zeta_n+1} and~\eqref{eqn:zeta_n+2}, which gives:
\begin{align}
	\zeta_{n+3}
	&= 2+ \frac{(\zeta_{n+2}+2)}{\zeta_{n+2}} (\zeta_{n+1}^2-4) \notag\\
	&\simeq 2 + \frac{ -\frac{4\alpha}{\beta} t}{(-2-\frac{4\alpha}{\beta} t)} \Bigl( \frac{4}{(\frac\alpha\beta t)^2} -4 \Bigr) \notag\\
	&\simeq 2+ \frac8{(1+\frac{2\alpha}\beta t) (\frac\alpha\beta t)}, \label{eq:zeta_n+3}
	\end{align}
so that $2\zeta_{n+3}^{-1}\simeq \frac{3\alpha}{16\beta}t$.  A second application gives
\begin{align}
	\zeta_{n+4}
	&=2+ \Bigl(1+\frac2{\zeta_{n+3}}\Bigr)(\zeta_{n+2}^2-4) \notag\\
	&\simeq 2+ \Bigl( 1+\frac{3\alpha}{16\beta} t\Bigr) \Bigl( \bigl(-2-\frac{4\alpha}\beta t\bigr)^2- 4\Bigr) \notag\\
	&\simeq 2+ \frac{16\alpha}\beta t. \label{eq:zeta_n+4}
	\end{align}
Now suppose we are given $0<\delta'<\delta$. By reducing $\delta'$ if necessary we find from~\eqref{eq:zeta_n+4} that the map $\zeta_n\mapsto\zeta_{n+4}$ takes the side of the interval $|\lambda-\tlambda|=|t|<\delta'$ that lies in the non-gap interval, meaning $\frac\alpha\beta t <0$,  to an interval of the form $(2-\delta'',2)\subset(0,2)$.  At the same time, and again reducing $\delta'$ if necessary, we can assume from~\eqref{eq:zeta_n+3} that $|\zeta_{n+3}|>2$ on this interval.  But then Lemma~\ref{lem:pointsononeside} is applicable to $\zeta_{n+3}$ and $\zeta_{n+4}$ and we find there is $k$ so that $\zeta_{n+4+2k}$ has a root in the interval. Since this argument was applicable to any $0<\delta'<\delta$ we conclude that the roots of the rational functions $\zeta_{n+2k}$ accumulate to $\tlambda$ as $k\to\infty$ within the non-gap interval.
\end{proof}

\begin{corollary}\label{cor:KNSCantor}
The support of the KNS spectrum is a Cantor set. In particular it is uncountable and has countably many gaps.
\end{corollary}
\begin{proof}
Recall from Corollary~\ref{cor:KNSspectfromDirNeum} that the support of the KNS spectral measure is the closure of the union of the set of Dirichlet  Laplacian eigenvalues on $G_n$. For $\tlambda$ a Dirichlet eigenvalue there is a least $n$ for which it is such, and the definition of $\zeta_n$ ensures $\zeta_n(\tlambda)=0$. But then Theorem~\ref{thm:valuesandgaps} provides a sequence $k_j$ and roots of $\zeta_{n+2k_j}$ that accumulate at $\tlambda$.  This shows each Dirichlet eigenvalue  for $G_n$ is a limit point of such eigenvalues, and therefore the support of the KNS spectrum is perfect.

If there was an interval in  the support of the KNS spectrum then by Corollary~\ref{cor:KNSspectfromDirNeum} it would contain an interior point $\tilde{\lambda}$ from the Dirichlet spectrum on some $G_n$. By assuming $n$ is the first index for which the eigenvalue $\tilde{\lambda}$ occurs we have $\zeta_n(\tilde{\lambda})=0$, so Theorem~\ref{thm:valuesandgaps} provides a gap on one side of $\tilde{\lambda}$ and we have a contradiction.  Accordingly the connected components of the support of the KNS spectrum  are points and the set is totally disconnected.

We have shown that the support of the KNS spectrum  is perfect and totally disconnected, so it is a Cantor set.
\end{proof}

The construction in the proof of Theorem~\ref{thm:valuesandgaps} allows us to find specific gaps by taking preimages of regions that the theorem ensures will escape under the dynamics~\eqref{eq:rootdynamics} and will therefore not contain eigenvalues. One can visualize these dynamics using graphs in $\mathbb{R}^2$, with coordinates $x=\zeta_2$ and $y=\zeta_3$.  We are interested only in those values that are given by~\eqref{eq:zetainitdata}, which are shown as thick curves on the graphs in Figure~\ref{fig:zetadynamics}.  The graph also shows the preimages of the escape region from Theorem~\ref{thm:valuesandgaps} for small $n$.  More precisely, these sets are where both $|\zeta_{n-2}|>2$ and $|\zeta_{n-1}|>2$.
Note that the intersections of the shaded regions with the thick curves correspond to intervals of $\lambda\in\mathbb{R}$ which cannot contain spectral values for any larger $n$, and are therefore gaps in the spectrum of $\Delta_n$ for all $n$. Using~\eqref{eq:rootdynamics} it is fairly easy to determine the endpoints of the intervals for any specified $n$.   If it were possible to give good estimates for the sizes of these intervals one could resolve the following question.
\begin{problem}
Determine whether the closure of the union of the spectra of the $L_n$ has zero Lebesgue measure or give estimates for its Hausdorff dimension.
\end{problem}

\begin{figure}
\centering
\includegraphics[width=3.5cm]{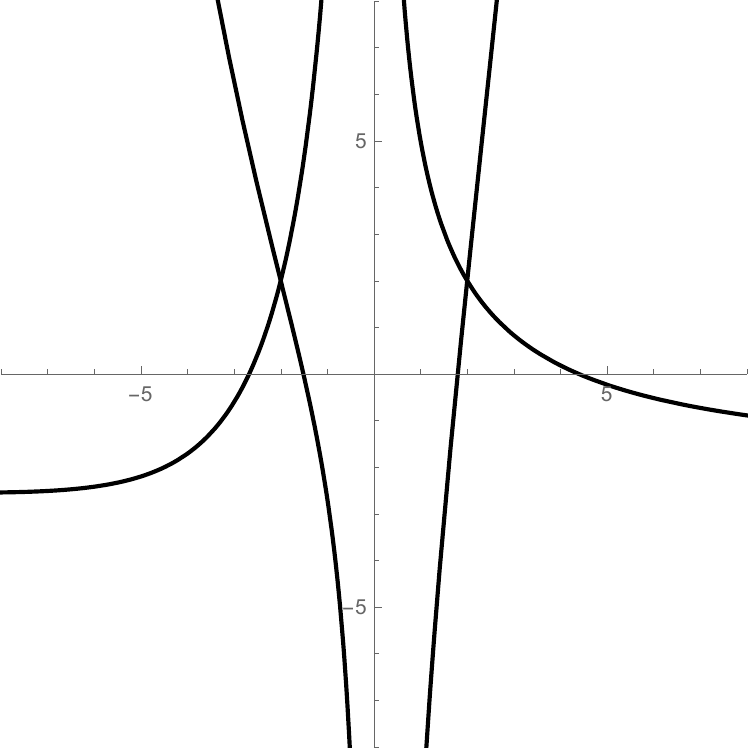} \hspace{.5cm}
\includegraphics[width=3.5cm]{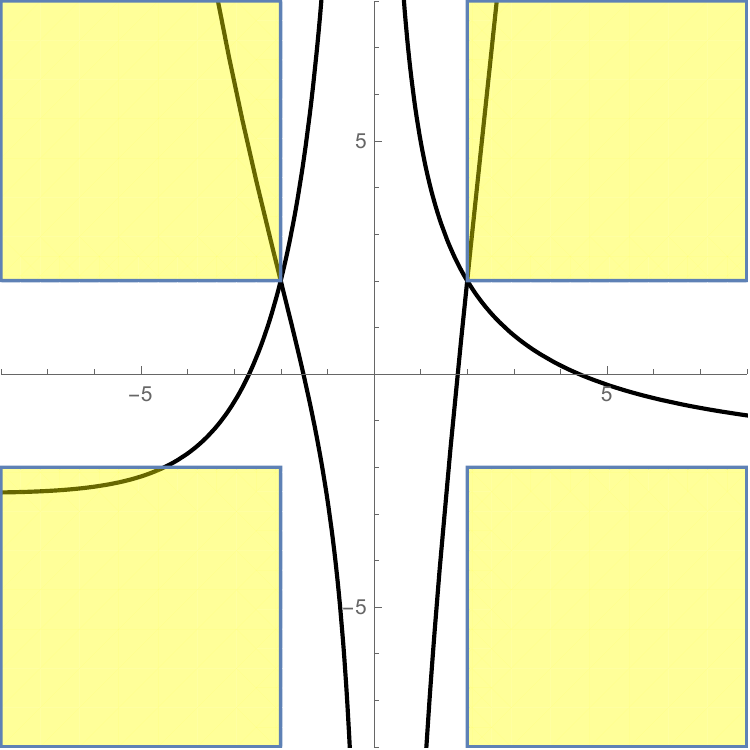} \hspace{.5cm}
\includegraphics[width=3.5cm]{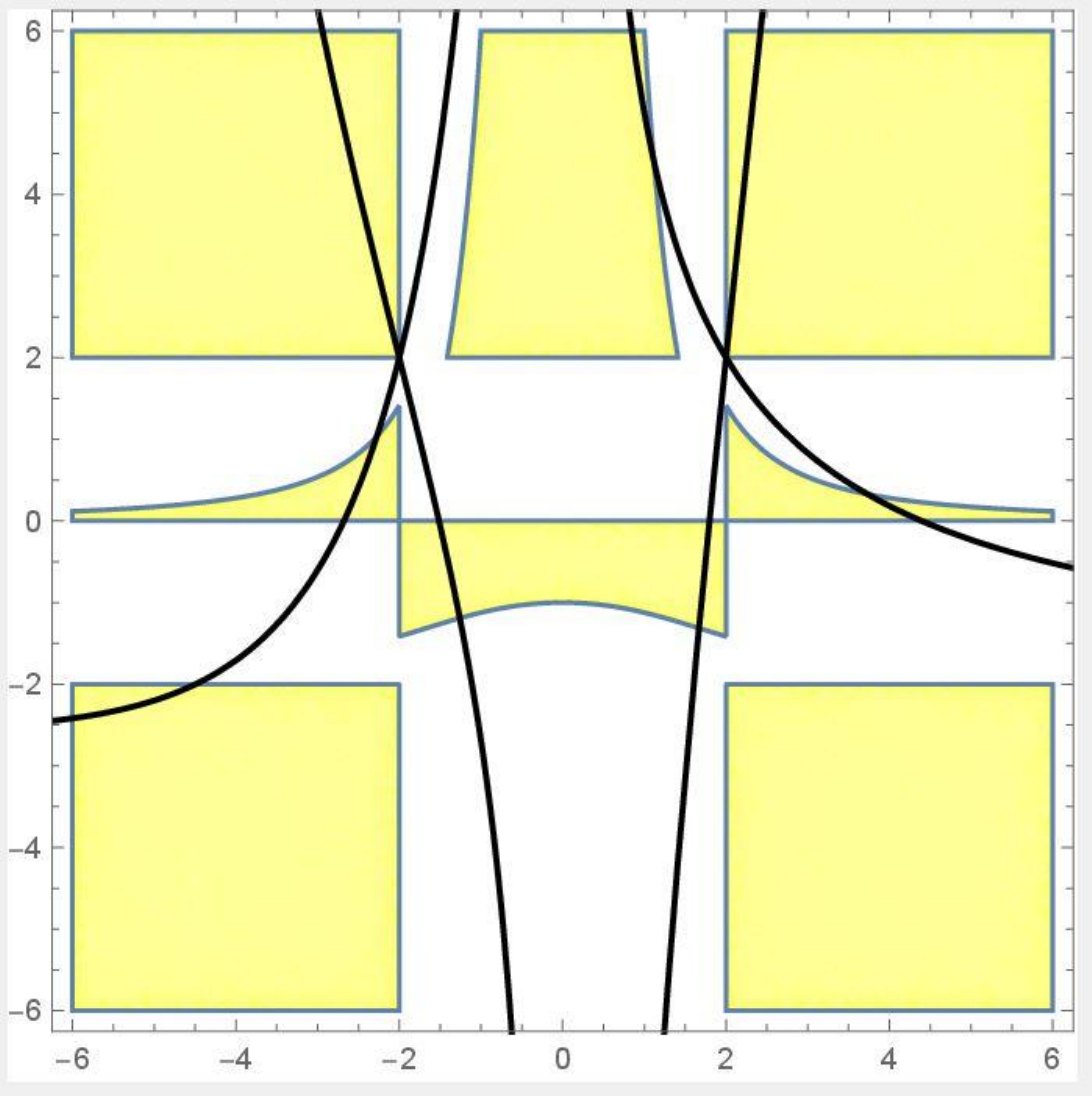} 
\caption{Graphs of $(\zeta_1(\lambda),\zeta_2(\lambda))$ (left) superimposed on escape regions $|\zeta_2|>2,|\zeta_3|>2$ (middle) and $|\zeta_4|>2,|\zeta_5|>2$ (right).}\label{fig:zetadynamics}
\end{figure}

\section{A generic set of blowups of the graphs $G_n$ with Pure Point Spectrum}\label{sec-pp}
Recall from Definition~\ref{def:Gnblowups} that a blow-up $G_\infty$ is the direct limit of a system $(G_{k_n},\iota_{k_n})$ with canonical graph morphisms $\tilde{\iota}_{k_n}:G_{k_n}\to G_\infty$ and the Laplacian $L_\infty$ on $G_\infty$ (from Definition~\ref{defn:Lapinfty}) at $\tilde{\iota}_{k_n}(x)$ for a non-boundary point $x\in G_{k_n}$ coincides with $L_{k_n}$ on $\tilde{\iota}_{k_n}(G_{k_n})$, as in~\eqref{eq:LinftyisLkn}.  We will write $\tilde{G}_{k_n}=\tilde{\iota}_{k_n}(G_{k_n})$ for the canonical copy of $G_{k_n}$ in $G_\infty$.

For the following lemma, note that $\tilde{\iota}_{k_n}$ can fail to be injective at the boundary points of $G_{k_n}$, but  $f\circ\tilde{\iota}_{k_n}^{-1}$ is well-defined for a Dirichlet eigenfunction $f$ because $f=0$ at the boundary points.
\begin{lemma}\label{lem:DNefnsextendtoGinfty}
If $f$ is a Dirichlet-Neumann eigenfunction of $L_{k_n}$ on $G_{k_n}$ then setting $F=f\circ\tilde{\iota}_{k_n}^{-1}$ on $\tilde{G}_{k_n}$ and zero elsewhere defines an eigenfunction of $L_\infty$ with the same eigenvalue and infinite multiplicity.
\end{lemma}
\begin{proof}
Let $\lambda$ be the eigenvalue of $L_{k_n}$ corresponding to $f$.
Using~\eqref{eq:LinftyisLkn} we have immediately that
\begin{equation}\label{eq:Linftyefneqn}
L_\infty F(\tilde{\iota}_{k_n}(x))=L_{k_n}f(x)=\lambda f(x)=\lambda F(\tilde{\iota}_{k_n}(x))
\end{equation}
if $x$ is not a boundary point of $G_{k_n}$.  If $x$ is a boundary point of $G_{k_n}$ then $\tilde{\iota}_{k_n}(x)$ may have neighbors in $G_\infty$ that are outside $\tilde{G}_{k_n}$, but since $F$ vanishes at these points we still have $L_\infty F(\tilde{\iota}_{k_n}(x))=L_{k_n}f(x)$ and therefore~\eqref{eq:Linftyefneqn} is still valid.  It remains to see $L_\infty F(y)=\lambda F(y)$ for $y\not\in \tilde{G}_{k_n}$, but for such $y$ we have $L_\infty F(y)=0=\lambda F(y)$  because $F$ vanishes at $y$ and its neighbors; some of these neighbors may be in $\tilde{G}_{k_n}$, in which case the fact that $F$ vanishes uses the Dirichlet property of $f$.  The corresponding eigenvalue has infinite multiplicity simply because there are an infinite number of distinct copies of any $G_m$ in $G_\infty$
\end{proof}

The eigenvalues coming from Dirichlet-Neumann eigenfunctions not only have infinite multiplicity. According to  Theorem~\ref{thm:howtogetmostofspect} they support an arbitrarily large proportion of the KNS spectral mass of $L_\infty$.  Even more is true for a certain class of blowups, for which we can show that spectrum is pure-point, with the set of Dirichlet-Neumann eigenfunctions generated at finite scales having dense span in $l^2$. Our proof closely follows an idea used to prove similar results for blow-ups of two-point self-similar graphs and Sierpinski Gaskets~\cite{MT,Teplyaevthesis}.

\begin{definition}\label{def:asymsubsps}
The subspace $l^2_a\subset l^2$ consists of the finitely supported functions that are antisymmetric in the following sense. The function $f\in l^2_a$ if there is $n$ such that $k_n-k_{n-1}=1$, $f$ is supported on $\tilde{\iota}_{k_{n-1}}(G_{k_n-1})$, and $g=f\circ\tilde{\iota}_{k_n}$ on  $G_{k_n}$ satisfies $g=-g\circ\Phi_{k_n}$.  See Figure~\ref{fig:asymsubsps}.
\end{definition}

\begin{figure}
\centering
\begin{tikzpicture}
\draw  (4.5,3) node[circle,fill,inner sep=2pt]{} -- (6,3) node[circle,fill,inner sep=2pt](300a){} --  (7.5,3) node[circle,fill,inner sep=2pt](3u){} --  (9,3) node[circle,fill,inner sep=2pt](300b){} --(10.5,3) node[circle,fill,inner sep=2pt]{};
\path[-, every loop/.style= {looseness=10, distance=20, in=300,out=240}]   (3u)  --   (7.5,2) node[circle,fill,inner sep=2pt](311){} -- (7.5,1) node[circle,fill,inner sep=2pt](3101){} edge [ultra thick, loop below] node {}(7.5,0)  ;
\path[every loop/.style= {looseness=10, distance=20, in=300,out=240}]   (300a) edge [loop below] node {} ()  (300b) edge [loop below] node {} ();
\draw[ultra thick] (3u) edge[bend right] node[left]{} (311) edge[ bend left] node[right]{} (311);
\draw[ultra thick] (311) edge[bend right] node[left]{$-h$} (3101) edge[ bend left] node[right]{$h$} (3101);
\end{tikzpicture}
\caption{A function from $l^2_a$ is supported and antisymmetric on a copy of $G_{k_{n-1}}$ in $G_{k_n}$ with $k_n-k_{n-1}=1$.}\label{fig:asymsubsps}
\end{figure}
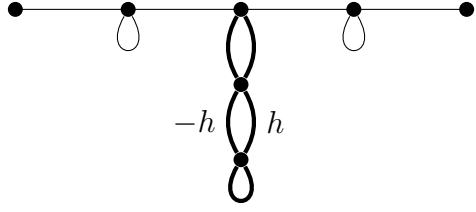

\begin{lemma}\label{lem:l2ainvariance}
The space $l^2_a$ is  invariant under $L_\infty$. Any eigenfunction of the restriction of  $L_\infty$ to $l^2_a$ is also an eigenfunction of $L_\infty$ and the corresponding eigenvalue has infinite multiplicity.  Moreover  $l^2_a$ is contained in the span of the finitely supported eigenfunctions of $L_\infty$.
\end{lemma}
\begin{proof}
The invariance is evident from the fact that $L_{k_n}$ is symmetric under $\Phi_{k_n}$ for each $n$ and~\eqref{eq:LinftyisLkn}.  Suppose $f$ is an eigenfunction of the restriction of $L_\infty$ to $l^2_a$.  Then there is $n$ as in Defintion~\ref{def:asymsubsps}, meaning $g=f\circ\tilde{\iota}_{k_n}$ satisfies $g=-g\circ\Phi_{k_n}$ and $g$ is supported on the copy of $G_{k_n-1}$ in $G_{k_n}$.  It follows from Theorem~\ref{thm:structure of DN efns} that $g$ is a Dirichlet-Neumann eigenfunction on $G_{k_n}$, and applying Lemma~\ref{lem:DNefnsextendtoGinfty} shows $f$ is an eigenfunction of $L_\infty$ and the eigenvalue has infinite multiplicity.

Now any function in $l^2_a$ has the structure described in Definition~\ref{def:asymsubsps} and is therefore in the span of the Dirichlet-Neumann eigenfunctions of $L_{k_n}$ for the $n$ given in that definition, and as was just mentioned,  Lemma~\ref{lem:DNefnsextendtoGinfty} provides that these extend to $G_\infty$ by zero to give finitely supported eigenfunctions of $L_\infty$.
\end{proof}

\begin{theorem}\label{thm:pureptspect}
If the blowup $(G_{k_n},\iota_{k_n})$ is such that both $k_{n+1}-k_n=1$ and $k_{n+1}-k_n=2$ occur for infinitely many $n$ then the antisymmetric subspace $l^2_a$ is dense in $l^2$.  Hence  there is an eigenbasis of finitely-supported antisymmetric eigenfunctions and the spectrum of $L_\infty$ is pure point.
\end{theorem}
\begin{proof}
Suppose $f\perp l^2_a$.
It will be useful to have some notation for the various subsets, subspaces and functions we encounter.  For fixed $n<m<\infty$ let us write $\iota'_{k_n,k_m}=\iota_{k_{m-1}}\circ\dotsm\circ\iota_{k_n}:G_{k_n}\to G_{k_m}$ and $G_{k_n}'=\iota'_{k_n,k_m}(G_{k_n}\setminus\partial G_{k_n})$ for the image of $G_{k_n}$, less its boundary points, in $G_{k_m}$ and $G_{k_n}''=\tilde{\iota}_{k_n}(G_{k_n}\setminus\partial G_{k_n})$ for the corresponding image in $G_\infty$.   We will write $P_n''f$ for the restriction of $f$ to $G_{k_n}''$, and $P_n'f=P_n''f\circ\tilde{\iota}_{k_m}$ for the corresponding function on $G_{k_m}$.  We frequently use the fact that, under counting measure, the integral of a function supported on $\tilde{G}_{k_n}$ may also be computed on $G_{k_n}$ or $G_{k_m}$.

The argument proceeds as follows. Since $f\in l^2$ we can take $n$ so large that $\|P_n''f\|_2\geq \frac23\|f\|_2$.  Using the hypothesis, we choose $m>n$ so that $k_m-k_{m-1}=1$ and there are $n<n'<n''<m$ with $k_{n'}-k_{n'-1}=1$ and $k_{n''}-k_{n''-1}=2$.  This choice ensures that $P_n''f$ vanishes at the point where  $\tilde{\iota}_{k_m}$ is non-injective, so setting $g =  P_n'f  - P_n'f\circ \Phi_{k_m}$ and $F=g\circ\tilde{\iota}_{k_m}^{-1}$ gives a well-defined function on $\tilde{G}_{k_m}\subset G_\infty$ that is antisymmetric in the sense of Definition~\ref{def:asymsubsps} and hence in $l^2_a$.    From this, and $f\perp l^2_a$, we may compute
\begin{align*}
	0&= \langle f,F\rangle_{l^2} = \bigl\langle f\circ\tilde{\iota}_{k_m}, g \bigr\rangle_{l^2_{k_m}}\\
	&=\langle f\circ\tilde{\iota}_{k_m}, P_n'f \rangle_{l^2_{k_m}} - \langle f\circ\tilde{\iota}_{k_m}, P_n'f \circ \Phi_{k_m}\rangle_{l^2_{k_m}}\\
	&= \langle f, P_n''f \rangle_{l^2} - \langle f\circ\tilde{\iota}_{k_m}\circ\Phi_{k_m}, P_n'f \rangle_{l^2_{k_m}}\\
	&=\|P_n''f\|_{l^2}^2 - \langle f|_{ \tilde{\iota}_{k_m}\circ \Phi_{k_m}(G_{k_n}')} , P_n''f \rangle_{l^2}
	\end{align*}
However our choice of $m$ also ensures that $\Phi_{k_m}(G_{k_n}')$ does not intersect $G_{k_n}'$ and thus $\tilde{\iota}_{k_m}\circ \Phi_{k_m}(G_{k_n}')$ does not intersect $G_{k_n}''$, so the restriction of $f$ to the former set has $l^2$ norm at most $\|f-P_n''f\|_{l^2}\leq\frac13\|f\|_2$.  By the above computation, the Cauchy-Schwartz inequality, and $\|P_n''f\|_{l^2}\geq \frac23\|f\|_{l^2}$ from our choice of $n$, we obtain
\begin{equation*}
	0\geq \|P_n''f\|_{l^2}^2 - \|P_n''f\|_{l^2}\|f-P_n''f\|_{l^2}
	\geq \frac49\|f\|_{l^2}^2 - \frac13 \|f\|_{l^2}^2
	= \frac19 \|f\|_{l^2}^2
	\end{equation*}
so that any $f\perp l^2_a$ is zero and thus $l^2_a$ is dense in $l^2$.  The remaining conclusions come from Lemma~\ref{lem:l2ainvariance}.
\end{proof}

Since the KNS spectrum is the limit of the spectra of the finitely supported eigenfunctions it follows immediately that the KNS spectrum is that of $L_\infty$.  The spectrum of $L_\infty$ is sometimes called the Kesten spectrum.

It is not difficult to use the condition on the sequence $\{k_n\}$ in Theorem~\ref{thm:pureptspect} and the description of the maps $\iota_{k_n}$ in Definition~\ref{def:Gnblowups} to determine the corresponding class of orbital Schreier graphs from  Theorem~\ref{thm:orbitalSareblowups}  for which Theorem~\ref{thm:pureptspect} guarantees the Laplacian spectrum is pure point.

\begin{corollary}
The fractal blowups identified as having pure point spectrum in  Theorem~\ref{thm:pureptspect} are all orbital Schreier graphs with one end.  All orbital Schreier graphs with one end have pure point spectrum with the possible exception of those isomorphic to $\Gamma_{\bar{1}}$.
\end{corollary}
\begin{proof}
Theorem~\ref{thm:pureptspect} applies to blowups for which  the values $1$ and $2$ both occur infinitely often in the sequence $\{k_{n+1}-k_{n}\}$.  When  $k_{n+1}-k_{n}=1$ then $\iota_{k_n}$ appends $1$ to non-boundary points and when $k_{n+1}-k_{n}=2$ it appends either $00$ or $01$.  Now observe that if $n$ and $n'$ are consecutive values such that $k_{n+1}-k_{n}=1=k_{n'+1}-k_{n'}$ then $k_{n+1}$ and $k_{n'+1}$ are of opposite parity; they cannot both be even or both be odd because the portion of the address between $v_{k_{n+1}}$ and $v_{k_{n'+1}}$ is a sequence made from $\{00,01\}$.  It follows that when we write $v=v_1v_2\dotsm$ the set $\{k: v_{2k}=1\}$ is infinite and so is $\{k:v_{2k+1}=1\}$.  What is more, if $n$ and $n'$ are consecutive values as before we see that, unless $n$ is the first such value, it must be that $v_{k_{n-1}}=0=v_{k_{n'-1}}$.  From this we deduce that both sets $\{k: v_{2k}=0\}$ and $\{k:v_{2k+1}=0\}$ are infinite.

Recall from the proof of Theorem~\ref{thm:orbitalSareblowups}  that an infinite blowup of the preceding type corresponds to the orbital Schreier graph $(\Gamma_v,v)$.  What is more, Theorem~4.1 of~\cite{nagnibeda} identifies the orbital Schreier graphs with one end as precisely those for which both  $\{k: v_{2k}=1\}$ and $\{k:v_{2k+1}=1\}$ are infinite sets.  We conclude that the fractal blowups to which Theorem~\ref{thm:pureptspect} applies are orbital Schreier graphs with one end.

Now suppose $v$ corresponds to an orbital Schreier graph $(\Gamma_v,v)$ with one end. It is apparent that $v$ may be written using the letter combinations $\{00,01,1\}$, because strings containing an even number of zeros may be written as $(00)^j$ and those with an odd number of zeros as $(00)^j(01)$; all remaining digits are copies of $1$.  If the whole sequence were written using only $00$ and $01$ then $\{k:v_k=1\}$ would consist entirely of numbers with the same parity (all would be odd or all would be even), which is impossible because  for an orbital Schreier graph with one end both $\{k: v_{2k}=1\}$ and $\{k:v_{2k+1}=1\}$ are infinite sets.  The same argument applies if we prepend any finite word to one written using  only $00$ and $01$.  It follows that $\{k:v_k=0\}$ is infinite and $v$ corresponds to an orbital Schreier graph  with one end then we can apply   Theorem~\ref{thm:pureptspect}  and find the spectrum is pure point.

The remaining possibility for an orbital Schreier graph with one end is that $\{k;v_k=0\}$ is finite. In this case $v=w\bar{1}$ for some finite word $w$.  By Theorem~5.4(1) of~\cite{nagnibeda} all such graphs are isomorphic.  Theorem~\ref{thm:pureptspect} does not apply in this case.
\end{proof}

Since we know the orbital Schreier graph corresponding to $\bar{1}$ has a non-trivial global symmetry (which we may think of as the reflection $\Phi_0$), the following consequence is immediate.

\begin{corollary}\label{cor-key}
If the Schreier graph has one end, but does not have a global symmetry, then 
we are in the situation of 
	 the generic set of blowups specified in Theorem~\ref{thm:pureptspect}. 
	 In this case  the spectrum of $L_\infty$ is pure point, and the set of eigenvalues coincides 
	 with the set of atoms of the KNS measure. 
	 In particular, in this case 
	 the spectrum of the Laplacian is the same as the support of the KNS measure, which is a Cantor set 
    by Corollary~\ref{cor:KNSCantor}.
\end{corollary}

\bibliography{BasilicaRefs}
\bibliographystyle{plain}

\end{document}